\documentclass[12pt]{amsart}
\usepackage{amssymb,amsmath,amsthm,dsfont}
\usepackage{mathrsfs}
\usepackage{mathtools}
\usepackage{mdwlist}
\usepackage{graphicx}
\usepackage{a4wide}
\usepackage{verbatim}
\usepackage{enumitem}
\input xy
\xyoption{all}

\newtheorem{theorem}{Theorem}[section]
\newtheorem{proposition}[theorem]{Proposition}
\newtheorem{lemma}[theorem]{Lemma}
\newtheorem{corollary}[theorem]{Corollary}

\theoremstyle{definition}
\newtheorem{definition}[theorem]{Definition}

\newtheorem{remark}[theorem]{Remark}

\numberwithin{equation}{section}

\newcommand{\R}{\mathbb{R}}

\newcommand{\N}{\mathbb{N}}

\newcommand{\Ca}{\mathcal{C}}
\newcommand{\A}{\mathcal{A}}
\newcommand{\F}{\mathscr{F}}
\newcommand{\G}{\mathscr{G}}
\newcommand{\K}{\mathcal{K}}
\newcommand{\M}{\mathcal{M}}
\newcommand{\Pa}{\mathcal{P}}
\newcommand{\Ra}{\mathcal{R}}
\newcommand{\Za}{\mathcal{Z}}

\newcommand{\e}{\varepsilon}
\newcommand{\w}{\widetilde}

\newcommand{\om}{\omega}
\newcommand{\Om}{\Omega}
\newcommand{\n}[1]{\| #1\|}
\newcommand{\nn}[1]{\vert\vert\vert #1\vert\vert\vert}
\newcommand{\abs}[1]{\vert #1\vert}
\newcommand{\ind}{\mathds{1}}
\newcommand{\ws}{weak$^\ast$}

\renewcommand{\span}{\mathrm{span}}
\newcommand{\supp}{\mathrm{supp}}
\newcommand{\SVM}{\mathsf{SVM}}
\newcommand{\Ext}{\mathrm{Ext}}
\newcommand{\JL}{\mathrm{JL}}
\newcommand{\id}{\mathrm{id}}
\newcommand{\dist}{\mathrm{dist}}
\newcommand{\cf}{\mathrm{cf}}
\renewcommand{\geq}{\geqslant}
\renewcommand{\leq}{\leqslant}
\newcommand{\striangle}{\scalebox{0.70}{$\triangle$}}
\newcommand{\sprod}{\scalebox{0.90}{$\prod$}}

\newcommand{\ex}[3]{0\to #1\to #2\to #3\to 0}
\newcommand{\exi}[5]{
0\longrightarrow #3\overset{#1}{\longrightarrow} #4\overset{#2}{\longrightarrow} #5\longrightarrow 0
}

\begin{document}


\title[Stability of vector measures and twisted sums of Banach spaces]{Stability of vector measures and twisted sums\\ of Banach spaces}

\author[T. Kochanek]{Tomasz Kochanek}
\address{Institute of Mathematics, University of Silesia,
Bankowa 14, 40-007 Katowice, Poland}
\email{t\_kochanek@wp.pl}

\begin{abstract}
A Banach space $X$ is said to have the $\SVM$ (stability of vector measures) property if there exists a~constant $v<\infty$ such that for any algebra of sets $\F$, and any function $\nu\colon\F\to X$ satisfying $$\n{\nu(A\cup B)-\nu(A)-\nu(B)}\leq 1\quad\mbox{for disjoint }A,B\in\F ,$$there is a~vector measure $\mu\colon\F\to X$ with $\n{\nu(A)-\mu(A)}\leq v$ for all $A\in\F$. If this condition is valid when restricted to set algebras $\F$ of cardinality less than some fixed cardinal number $\kappa$, then we say that $X$ has the $\kappa$-$\SVM$ property. The least cardinal $\kappa$ for which $X$ does not have the $\kappa$-$\SVM$ property (if it exists) is called the $\SVM$ character of $X$. We apply the machinery of twisted sums and quasi-linear maps to characterise these properties and to determine $\SVM$ characters for many classical Banach spaces. We also discuss connections between the $\kappa$-$\SVM$ property, $\kappa$-injectivity and the `three-space' problem.
\end{abstract}

\maketitle

%%%%%%%%%%%%%%%%%%%%%%%%%%%%%%%%%%%%%%%%%
\section{Introduction}
\noindent
It is a widely recognised fact that there is a strong interplay between the vector measure theory and the Banach space theory. Many structural properties of Banach spaces (such as being isomorphic to a dual space, containing an isomorphic copy of $c_0$ or $\ell_\infty$, sequential completeness, admitting a boundedly complete basis {\it etc}.) translate into natural properties of vector measures. Several representation theorems involving vector measures provide also important results concerning the form of operators of various types (weakly compact, nuclear, absolutely summing {\it etc}.)~acting on Banach spaces like $\Ca(K)$ or $L_\infty(\mu)$. The classical monograph \cite{diestel_uhl} by Diestel and Uhl gives an excellent general treatment of the subject.

In this paper, we will deal with a~stability problem for vector measures (by a {\it vector measure} we will always mean a~finitely additive set function) and discuss its connections with the theory of twisted sums of Banach spaces. The following two definitions will be of central importance to us:
\begin{definition}\label{def_SVM}
We say that a Banach space $X$ has the $\SVM$ ({\it stability of vector measures}) {\it property} if there exists a~constant $v(X)<\infty$ (depending only on $X$) such that given any set algebra (that is, any family of subsets of some set containing $\varnothing$ and closed under complementation and finite unions) $\F$ and any mapping $\nu\colon\F\to X$ satisfying
\begin{equation}\label{1}
\n{\nu(A\cup B)-\nu(A)-\nu(B)}\leq 1\quad\mbox{for }A,B\in\F ,\, A\cap B=\varnothing,
\end{equation}
there is a vector measure $\mu\colon\F\to X$ with $$\n{\nu(A)-\mu(A)}\leq v(X)\quad\mbox{for }A\in\F.$$
\end{definition}
\begin{definition}\label{def_kSVM}
Let $\kappa$ be a cardinal number. We say that a Banach space $X$ has the $\kappa$-$\SVM$ {\it property} if there exists a~constant $v(\kappa,X)<\infty$ (depending only on $\kappa$ and $X$) such that given any set algebra $\F$ with cardinality less than $\kappa$, and any mapping $\nu\colon\F\to X$ satisfying \eqref{1}, there is a~vector measure $\mu\colon\F\to X$ with $$\n{\nu(A)-\mu(A)}\leq v(\kappa, X)\quad\mbox{for }A\in\F.$$

If $X$ is a Banach space which does not have the $\SVM$ property, then by the $\SVM$ {\it character} of $X$ we mean the minimal cardinal number $\kappa$ such that $X$ does not have the $\kappa$-$\SVM$ property, and we denote it $\tau(X)$.
\end{definition}

Any function $\nu$ satisfying inequality \eqref{1} with some $\e$ instead of $1$ will be called $\e$-{\it additive}. If $A$ is a~set then $\abs{A}$ stands for its cardinality; if $\kappa$ is a~cardinal number then $\kappa^+$ stands for its cardinal successor. Note that the definition of $\SVM$ character is well-posed, since every Banach space $X$ without the $\SVM$ property does not have the $\kappa$-$\SVM$ property for some cardinal number $\kappa$. To see this, take a~sequence of $1$-additive maps $\nu_n\colon\F_n\to X$ such that for each $n\in\N$ and every vector measure $\mu\colon\F_n\to X$ we have $\n{\nu_n(A)-\mu(A)}\geq n$ for some $A\in\F_n$, and consider the simple product $\Sigma=\sprod_n\F_n$ of Boolean algebras $(\F_n)_{n=1}^\infty$. Then, the maps $\nu_n\circ\pi_n\colon\Sigma\to X$ (where $\pi_n$ is the projection onto the $n$th coordinate) are all $1$-additive and witness that $X$ does not have the $\abs{\Sigma}^+$-$\SVM$ property.

In general, by saying $\tau(X)>\kappa$ (or $\geq\kappa$) we formally mean that either $X$ does not have the $\SVM$ property and then the claimed inequality holds true, or simply $X$ has the $\SVM$ property. In the case where $X$ has the $\kappa$-$\SVM$ property the symbol $v(\kappa,X)$ will stand for the infimum of all such numbers for which the condition from Definition \ref{def_kSVM} is valid (it may happen that this condition is no longer true for the infimum).

The origin of this type of stability notion is traced back to Ulam (see \cite{ulam_1}) who in 1940 formulated the following problem: Let $G_1$ be a~group and let $G_2$ be a~metric group with a~metric $d(\cdot ,\cdot)$. Given $\delta>0$, does there exist $\e>0$ such that to each mapping $F\colon G_1\to G_2$ satisfying $d(F(x+y),F(x)+F(y))\leq\e$ for all $x,y\in G_1$ there corresponds a homomorphism $A\colon G_1\to G_2$ with $d(F(x),A(x))\leq\delta$ for all $x\in G_1$? It should be mentioned that a version of Ulam's problem for real sequences appeared in the book of P\'olya and Szeg\H{o} \cite[Chapter 3, Problem 99]{polya_szego}.  Hyers \cite{hyers} was the first who, in 1941, gave a~solution in the case where $G_1$ and $G_2$ are Banach spaces. Nowadays the theory of Hyers--Ulam stability is widely developed; one can consult, {\it e.g.}, the survey paper \cite{forti} by Forti.

Kalton and Roberts were the first who were dealing with a~stability problem for (real-valued) set additive functions, and in 1983 they published the following, beautiful result which is the main motivation for our study.
\begin{theorem}[Kalton \& Roberts \cite{kalton_roberts}]\label{KR}
There is an absolute constant $K<45$ having the property: If $\F$ is a~set algebra and a~function $\nu\colon\F\to\R$ satisfies $$\abs{\nu(A\cup B)-\nu(A)-\nu(B)}\leq 1\quad\mbox{for }A,B\in\F ,\, A\cap B=\varnothing,$$then there exists an additive set function $\mu\colon\F\to\R$ such that $\abs{\nu(A)-\mu(A)}\leq K$ for $A\in\F$.
\end{theorem}
From now on $K$ will stand for the best possible constant in the above theorem and will be referred to as the {\it Kalton--Roberts constant}. The exact value of $K$ is not known; the lower bound $K\geq 3/2$ was shown by Pawlik \cite{pawlik}. In our terminology, Theorem \ref{KR} asserts that the one-dimensional space $\R$ has the $\SVM$ property and, as an immediate consequence, all the finite-dimensional spaces $\R^n$, as well as the space $\ell_\infty$, also have this property (it does not matter which concrete norm in $\R^n$ we decide to use; it may only affect the value of $v(\R^n)$). The question whether a nearly additive real-valued set function may be approximated by an additive one was posed by Kalton, explicitly in \cite{kalton (problem)} and implicitly in \cite{kalton}. The reason for which he found it of interest is that it is a~kind of reformulation of the question whether $c_0$ and $\ell_\infty$ are the so-called $\K$-{\it spaces}, {\it i.e.} whether they do not admit any non-trivial exact sequences of the form $\ex{\R}{Z}{c_0}$, or $\ex{\R}{Z}{\ell_\infty}$. We refer the reader to Section 3 for more technical details. Now, let us give a~brief outline of the background to these issues.

The so-called {\it three-space problem} (shortly: $\mathsf{3SP}$ problem, cf. \cite{castillo_gonzalez}) lies at the heart of the Banach space theory. Its general framework is the following: Given an exact sequence $\ex{Y}{Z}{X}$ of Banach spaces (more generally, $F$-spaces) $X$, $Y$, $Z$, which properties of the middle space $Z$ follow from the analogous properties of $X$ and $Y$? A~property $\Pa$ is called a~$\mathsf{3SP}$ {\it property} (in some prescribed class of linear topological spaces) whenever the implication: $X$ and $Y$ posses $\Pa$, then $Z$ also possesses $\Pa$, holds true for any exact sequence as above built from spaces belonging to the class considered. The $\mathsf{3SP}$ problem became one of central topics in functional analysis in seventies, when Enflo, Lindenstrauss and Pisier \cite{enflo_lindenstrauss_pisier} found a counterexample to the problem of Palais: If $X$ and $Y$ are Hilbert spaces, must $Z$ be (isomorphic to) a~Hilbert space? Many other basic properties, like local convexity or being weakly compactly generated, have led to some remarkable examples in the theory of Banach spaces and $F$-spaces. The monograph \cite{castillo_gonzalez} by Castillo and Gonz\'alez exhibits many of them, as well as the general theory of the $\mathsf{3SP}$ problem.

One of the turning points in the history of the $\mathsf{3SP}$ problem occurred in late seventies when Kalton published his paper \cite{kalton} and then, jointly with Peck, \cite{kalton_peck}. Their main feature is that they give a~deep general link between the structure of exact sequences of (locally bounded) $F$-spaces and the stability behaviour of the so-called quasi-linear maps. Given two locally bounded $F$-spaces (equivalently: quasi-Banach spaces) $X$ and $Y$, we say that a~map $f\colon X\to Y$ is {\it quasi-linear} if it is homogeneous and satisfies $$\n{f(x+y)-f(x)-f(y)}\leq c(\n{x}+\n{y})\quad\mbox{for }x,y\in X,$$where $c<\infty$ is some constant. Kalton and Ribe showed that every exact sequence $$\exi{i}{q}{Y}{Z}{X}$$ is, in a sense, generated by some quasi-linear map $f\colon X\to Y$ and, moreover, such a~sequence splits ({\it i.e.} $Z\simeq X\oplus Y$ and, through some isomorphism, $i$ is the natural inclusion and $q$ is the canonical projection; see Section~3 for details) if and only if there exists a~linear (not necessarily continuous) map $h\colon X\to Y$ such that $$\n{f(x)-h(x)}\leq kc\n{x}\quad\mbox{for }x\in X,$$ where $k<\infty$ depends only on $X$ and $Y$. In other words, if we wish to construct a~non-trivial exact sequence with a~subspace isomorphic to $Y$ and a~quotient isomorphic to $X$, then we shall be looking for a~quasi-linear map from $X$ into $Y$ which is not approximable by any linear one. If no such quasi-linear map exists, then we have a~kind of stability effect for quasi-linear maps between $X$ and $Y$.

The aim of this paper is to show that some aspects of the $\mathsf{3SP}$ problem may be described in the language of vector measures and the structure of exact sequences of Banach spaces depends on stability properties of vector measures, likewise it depends on stability properties of quasi-linear maps. 

The paper is organised as follows: In Section 2 we give several very basic results concerning the $\kappa$-$\SVM$ property. Among others, we show that the Kalton--Roberts theorem almost immediately implies that $\tau(c_0(\Gamma))>\omega$ and $\tau(\mathcal{C}(\Omega))>\omega$ for every non-empty index set $\Gamma$ and every compact metric space $\Omega$ (Proposition \ref{omega_c0}) and, more generally, that $\tau(X)>\omega$ for every $\mathscr{L}_\infty$-space $X$ (Corollary \ref{linfty}). Section 3 reviews some of the background material needed in later sections. In Section 4 we present two theorems which give necessary conditions for the $\SVM$ property and the $\kappa$-$\SVM$ property. Theorem \ref{T2} says that the $\bigl(2^\Gamma\bigr)^+$-$\SVM$ property of $X$ implies that the pair $(\ell_\infty(\Gamma),X)$ splits, whereas the $\Gamma^+$-$\SVM$ property implies that $(c_0(\Gamma),X)$ splits. From this we may conclude that many classical Banach spaces, like $L_p$-spaces ($1\leq p<\infty$), quasi-reflexive James' space, Schreier's space, Tsirelson's space {\it etc}., have $\SVM$ character equal to $\omega$ (Corollary \ref{Cor2}). However, for compact metric spaces $\Omega$ the $\SVM$ character of $\mathcal{C}(\Omega)$ equals $\omega_1$, unless this space is finite-dimensional or isomorphic to $c_0$ (Corollary \ref{Cor3}). Next, in Section 5, we apply the machinery of twisted sums to prove (Theorem \ref{k_injective}) that $\kappa$-injectivity implies the $\kappa$-$\SVM$ property, provided $\kappa$ has uncountable cofinality; for other $\kappa$'s we have to assume $(\lambda,\kappa)$-injectivity, for some fixed number $\lambda$ (see Definition \ref{kappa}). Continuing this line of discussion, we define in Section 6 a~generalised Johnson--Lindenstrauss space, denoted $\JL_\infty(\Gamma)$, which will be used to derive the equality $\tau(c_0(\Gamma))=\omega_2$ (Corollary \ref{c0}). We also build a~non-trivial twisted sum of $\JL_\infty$ (the usual Johnson--Lindenstrauss space) and $c_0(\omega_1)$ and show that $\tau(\JL_\infty)=\omega_2$ (Proposition \ref{ext_JL} and Theorem \ref{JLchar}). In Section 7 we give a~characterisation of the $\kappa$-$\SVM$ property in a~form of the condition $\Ext(\mathfrak{X},\cdot)=0$, postulated for Banach spaces $\mathfrak{X}$ from a~certain class of `testing' spaces (Corollary \ref{cor_iff}). By doing so, we prove that the $\kappa$-$\SVM$ property is a~$\mathsf{3SP}$ property (Theorem \ref{3sp}). In Section 8 we characterise the $\SVM$ property for Banach spaces complemented in their biduals, {\it e.g.} by the condition $\Ext(X^\ast,\ell_1)=0$ (Theorem \ref{T4}). Finally, Section 9 contains a~list of open problems stemming naturally from the presented material.

Our Banach space, vector measure and set theory terminology will follow \cite{albiac_kalton}, \cite{diestel_uhl} and \cite{ciesielski}, respectively. By an {\it operator} we always mean a~continuous and linear mapping, and by a~{\it subspace} we mean a~{\it closed} linear subspace, of course not counting these cases when we talk about a~dense subspace. An {\it isomorphism} is a one-to-one surjective operator whose inverse is an operator as well. The fact that Banach spaces $X$ and $Y$ are isomorphic we note as $X\simeq Y$. We also consider all spaces to be real, since the complex case makes no serious difference in the theory.
%%%%%%%%%%%%%%%%%%%%%%%%%%%%%%%%%%%%%%%%%%%
%%%%%%%%%%%%%%%%%%%%%%%%%%%%%%%%%%%%%%%%%%%
\section{Preliminary results}
\noindent
In this section we collect several simple observations which, however, give a~good starting point for investigating the $\kappa$-$\SVM$ and the $\SVM$ properties, and which will be found useful later on. 

First, observe that every Banach space $X$ has the $n$-$\SVM$ property for each $n\in\N$. Indeed, if $\F\subset 2^\Om$ is a~finite algebra of sets (we may assume that $\F=2^\Om$), $\abs{\F}<n$, and $\nu\colon\F\to X$ is $1$-additive then, by a~simple induction, we have $$\Biggl\|\nu(A)-\sum_{a\in A}\nu\{a\}\Biggr\|\leq\abs{A}-1<\log_2n-1\quad\mbox{for }A\in\F,$$thus the measure $\mu\colon\F\to X$, defined by $\mu\{a\}=\nu\{a\}$ for $a\in\Om$, does the job. Consequently, for every Banach space $X$ we have $\tau(X)\geq\om$ and $v(n,X)<\log_2n-1$ for each $n\in\N$.

As a~warm-up, we prove a~lower estimate for the $\SVM$ characters of $c_0(\Gamma)$ and $\mathcal{C}(\Omega)$, with $\Omega$ being a~compact metric space, which is an~almost immediate consequence of the Kalton--Roberts theorem. This estimate will also follow from a~more general result, Corollary \ref{linfty}, but for the spaces: $c_0(\Gamma)$, $\mathcal{C}[0,\omega^{\omega^\alpha}]$, $\alpha$ being a~countable ordinal, and $\mathcal{C}[0,1]$ (the latter two exhausting the whole class of $\mathcal{C}(\Omega)$-spaces with compact, metrisable $\Omega$) the proof is more direct and an~approximating vector measure is defined coordinate-wise without appealing to the fact that the underlying space is a~$\mathscr{L}_\infty$-space.
\begin{proposition}\label{omega_c0}
For every non-empty set $\Gamma$ and every compact metric space $\Omega$ we have:
\begin{itemize*}
\item[{\rm (i)}] $\tau(c_0(\Gamma))>\omega$;
\item[{\rm (ii)}] $\tau(\mathcal{C}(\Omega))>\omega$.
\end{itemize*}
\end{proposition}
\begin{proof}
(i): Let $\F$ be any finite field of sets and let $\nu\colon\F\to c_0(\Gamma)$ be a $1$-additive function. Choose any $\e>0$ and pick a~finite set $F_\e\subset\Gamma$ such that $\abs{e_\gamma^\ast(\nu(A))}<\e$ for each $\gamma\in\Gamma\setminus F_\e$ and $A\in\F$, where $e_\gamma^\ast$ is the $\gamma$th coordinate functional on $c_0(\Gamma)$. For each $\gamma\in F_\e$ separately apply Theorem \ref{KR} to produce an additive set function $\mu_\gamma\colon\F\to\R$ satisfying $\abs{e_\gamma^\ast(\nu(A))-\mu_\gamma(A)}\leq K$ for $A\in\F$. Then the measure $\mu\colon\F\to c_0(\Gamma)$, defined by $$e_\gamma^\ast\mu(A)=\mu_\gamma(A)\,\,\,\mbox{for }\gamma\in F_\e,\mbox{ and }e_\gamma^\ast\mu(A)=0\,\,\,\mbox{for }\gamma\in\Gamma\setminus F_\e$$
satisfies $\n{\nu(A)-\mu(A)}\leq K$ for $A\in\F$. We have thus proved that $c_0(\Gamma)$ has the $\om$-$\SVM$ property and, moreover, $v(\om,c_0(\Gamma))=K$.

\vspace*{2mm}
(ii): In view of Miljutin's theorem (\cite{miljutin}, see also \cite[Theorem 2.1]{rosenthal (chapter)}), for every uncountable metric compact space $\Omega$ the space $\mathcal{C}(\Omega)$ is isomorphic to $\mathcal{C}[0,1]$, whereas the theorem of Bessaga and Pe\l czy\'nski (\cite{bessaga_pelczynski}, see also \cite[Theorem 2.14]{rosenthal (chapter)}) says that for every infinite countable compact metric space $K$ the space $\mathcal{C}(\Omega)$ is isomorphic to one of the spaces $\mathcal{C}[0,\omega^{\omega^\alpha}]$, $\alpha\geq 0$ being a~countable ordinal, which are mutually non-isomorphic. Therefore, in the rest of the proof we may (and do) assume that $\Omega$ is either $[0,1]$ or $[0,\omega^{\omega^\alpha}]$ for some countable ordinal $\alpha$.

Let $\F$ be any finite field of sets and suppose $\nu\colon\F\to \mathcal{C}(\Omega)$ is a~$1$-additive function. For any $\e>0$ there is a~finite $\e$-net $$0=t_0<t_1\ldots <t_n=\left\{\begin{array}{ll}1 & \mbox{if }\Omega=[0,1]\\ \omega^{\omega^\alpha} & \mbox{if }\Omega=[0,\omega^{\omega^\alpha}]\end{array}\right.$$ of the space $\Omega$ which may be also chosen so that for every $A\in\F$ the oscillation of the function $\nu(A)$ inside the interval $[t_{j-1},t_j]$ is at most $\e$, for each $1\leq j\leq n$. By applying Theorem \ref{KR} to each of the $1$-additive functions $\F\ni A\mapsto\nu(A)(t_j)$ ($0\leq j\leq n$), we get scalar measures $\mu_j$ such that $\abs{\nu(A)(t_j)-\mu_j(A)}\leq K$ for $A\in\F$ and $0\leq j\leq n$. Now, for every $A\in\F$ we define a~continuous function $\mu(A)$ interpolating the points $(t_j,\mu_j(A))$ ($0\leq j\leq n$) as follows:
\begin{itemize*}
\item if $\Omega=[0,1]$ we let $\mu(A)$ be the piecewise linear function, agreeing with $\mu_j(A)$ at $t_j$ for each $0\leq j\leq n$, and linear on all the intervals of $[0,1]\setminus\{t_0,\ldots ,t_n\}$;
\item if $\Omega=[0,\omega^{\omega^\alpha}]$ then for each $1\leq j\leq n$ we put $\mu(A)(t)=\mu_j(A)$ for all $t\in (t_{j-1}, t_j]$ and $\mu(A)(0)=\mu_0(A)$. 
\end{itemize*}

Plainly, for each $A\in\F$ the function $\mu(A)$ belongs to $\mathcal{C}(\Omega)$ and the so-defined mapping $\mu\colon\F\to \mathcal{C}(\Omega)$ is a~vector measure. Finally, for any $A\in\F$, $t\in K$ and $j\in\{0,1,\ldots ,n\}$ with $t\in [t_{j-1},t_j]$, we have $$\abs{\nu(A)(t)-\mu(A)(t)}\leq\abs{\nu(A)(t_j)-\mu_j(A)}+2\e\leq K+2\e,$$since the oscillations of the both functions $\nu(A)$ and $\mu(A)$ inside $[t_{j-1},t_j]$ is at most $\e$. This shows that $\n{\nu(A)-\mu(A)}\leq K+2\e$ for each $A\in\F$ and, consequently, $\mathcal{C}(\Omega)$ has the $\omega$-$\SVM$ property with $v(\omega,\mathcal{C}(\Omega))=K$.
\end{proof}

Although this was so simple, it is by no means obvious at the moment whether we can extend the result for countably infinite fields, that is, whether $c_0(\Gamma)$ and $\mathcal{C}(\Omega)$ have the $\om_1$-$\SVM$ property. It is not too helpful to know that for all finite algebras there is a~common constant, the Kalton--Roberts constant, which controls the distance from any $1$-additive function to the set of all vector measures. Since none of the infinite-dimensional spaces $c_0(\Gamma)$ and $\mathcal{C}(\Omega)$, with compact and metrisable $\Omega$, is complemented in its bidual (and therefore, complemented in any other Banach space), none of \ws\ topologies would make it possible to use standard compactness arguments. It turns out that the problem of determining $\tau(c_0(\Gamma))$ and $\tau(\mathcal{C}(\Omega))$ requires a~bit stronger machinery.

At this point let us note what can be deduced from the $\om$-$\SVM$ property via compactness arguments.
\begin{proposition}\label{bidual}
If a Banach space $X$ is $\lambda$-complemented in its bidual and has the $\om$-$\SVM$ property, then it has the $\SVM$\ property and $v(X)\leq\lambda v(\om,X)$.
\end{proposition}
\begin{proof}
Let $X$ be a Banach space complemented in $X^{\ast\ast}$ and having the $\omega$-$\SVM$ property. Pick any number $v>v(\om,X)$. Suppose that $\F$ is an arbitrary algebra of sets and a~function $\nu\colon\F\to X$ is $1$-additive. Let $\Gamma$ be the set of all finite subfields of $\F$, directed by inclusion. For each $\A\in\Gamma$ there is a~vector measure $\mu_\A\colon\A\to X$ such that $\n{\nu(A)-\mu_\A(A)}\leq v$ for every $A\in\A$. We extend each $\mu_\A$ to the whole of $\F$ by putting $\mu_\A(A)=0$ for $A\in\F\setminus\A$. Then the net $\{\mu_\A\}_{\A\in\Gamma}$ is contained in the set $$\bigl\{f\colon\F\to X^{\ast\ast}\,\,\vert\,\,\n{f(A)}\leq\n{\nu(A)}+v\mbox{ for each }A\in\F\bigr\},$$which is compact with respect to the product topology of the \ws\ topology on $X^{\ast\ast}$ (as usual, we identify $X$ with its canonical copy inside $X^{\ast\ast}$). Hence, there is a~subnet of $\{\mu_\A\}_{\A\in\Gamma}$ which is pointwise convergent to some function $\mu\colon\F\to X^{\ast\ast}$. Obviously, $\mu$ is a~vector measure and for each $A\in\F$ we have $\n{\nu(A)-\mu(A)}\leq v$. Finally, let $P\colon X^{\ast\ast}\to X$ be a~projection with $\n{P}\leq\lambda$. Then, since $P(x)=x$ for $x\in X$, we have $\n{\nu(A)-(P\circ\mu)(A)}\leq\lambda v$ for $A\in\F$, which completes the proof.
\end{proof}

Combining this result with our earlier observation, that always $\tau(X)\geq\omega$, we infer that for Banach spaces which are complemented in their biduals we have the following dichotomy: either their $\SVM$ character equals $\om$, or they have the $\SVM$ property. For other Banach spaces, like $c_0(\Gamma)$ and $\mathcal{C}(\Omega)$, we will see that everything may happen.

Now, let us note the following trivial but useful observation. The proof is straightforward, so we omit it.
\begin{proposition}\label{comp}
Let $X$ be a Banach space and $Y$ be a~$\lambda$-complemented subspace of $X$. If $X$ has the $\kappa$-$\SVM$ property then so does $Y$ and, moreover, $v(\kappa,Y)\leq\lambda v(\kappa,X)$. Consequently, if $X$ has the $\SVM$ property then so does $Y$, and we have $v(Y)\leq\lambda v(X)$.
\end{proposition}

\begin{theorem}\label{injective}
Every injective Banach space has the $\SVM$ property.
\end{theorem}
\begin{proof}
By applying Theorem \ref{KR} to each coordinate separately, we see that for every non-empty set $\Gamma$ the space $\ell_\infty(\Gamma)$ has the $\SVM$ property. Moreover, it is well-known that every Banach space $X$ embeds into $\ell_\infty(\Gamma)$, where the cardinality of $\Gamma$ is the topological density of $X$, and if $X$ is injective then it must be complemented in $\ell_\infty(\Gamma)$. Hence, the assertion follows from Proposition \ref{comp}.
\end{proof}

Now, we wish to derive a `local' version of Proposition \ref{comp} and give the promised generalisation of Proposition \ref{omega_c0}. To this end let us recall some definitions.

Let $\mathcal{E}$ be a family of finite-dimensional Banach spaces. We say that a~Banach space $X$ is $\lambda$-{\it locally} $\mathcal{E}$ (for some $\lambda>1$) provided that every finite-dimensional subspace of $X$ is contained in a~finite-dimensional subspace $F$ of $X$ satisfying $d_\mathsf{BM}(E,F)<\lambda$ for some $E\in\mathcal{E}$ with $\dim E=\dim F$, where $d_\mathsf{BM}$ stands for the {\it Banach--Mazur distance} defined by $$d_\mathsf{BM}(E,F)=\inf\bigl\{\n{T}\cdot\n{T^{-1}}\colon T\mbox{ is an isomorphism }E\to F\bigr\}.$$We say that $X$ is {\it locally} $\mathcal{E}$ if it is $\lambda$-locally $\mathcal{E}$ for some $\lambda\geq 1$. The concept of locality (and colocality) was studied in the context of twisted sums of Banach spaces by Jebreen, Jamjoom and Yost \cite{jebreen_jamjoom_yost}. In the case where $\mathcal{E}=\{\ell_p^n\}_{n=1}^\infty$ ($1\leq p\leq\infty$) this notion defines nothing else but the class of $\mathscr{L}_p$-spaces introduced by Lindenstrauss and Pe\l czy\'nski in \cite{lindenstrauss_pelczynski} and investigated also by Lindenstrauss and Rosenthal in \cite{lindenstrauss_rosenthal}. 

We also say that a Banach space $X$ {\it contains} $\mathcal{E}$ $(c,d)$-{\it uniformly complemented} (for some $c\geq 1$, $d>1$) provided that for each $E\in\mathcal{E}$ one may find a~$c$-complemented subspace $F$ of $X$ with $\dim E=\dim F$ and $d_\mathsf{BM}(E,F)<d$. We say that $X$ {\it contains} $\mathcal{E}$ {\it uniformly complemented} if the above condition is valid for some $c\geq 1$ and $d>1$.

\begin{proposition}\label{locally}
Let $\mathcal{E}$ be a~family of finite-dimensional Banach spaces and $X$ be a~Banach space that is $\lambda$-locally $\mathcal{E}$. If there is a~Banach space $Y$ containing $\mathcal{E}$ $(c,d)$-uniformly complemented and satisfying $\tau(Y)>\omega$, then $\tau(X)>\omega$ and $v(\omega,X)\leq \lambda cdv(\omega,Y)$.

Therefore, by contraposition, if $X$ is a~Banach space that is locally $\mathcal{E}$ and $\tau(X)=\omega$, then $\tau(Y)=\omega$ for every Banach space $Y$ containing $\mathcal{E}$ uniformly complemented.
\end{proposition}
\begin{proof}
For each $E\in\mathcal{E}$ we may pick a~$c$-complemented subspace $Y_E$ of $Y$, $\dim E=\dim Y_E$, and an~isomorphism $T_E\colon E\to Y_E$ such that $\n{y}\leq\n{T_Ey}\leq d\n{y}$ for every $y\in E$.

Let $\F$ be any finite set algebra and $\nu\colon\F\to X$ be a~$1$-additive function. Then, there exist a~finite-dimensional space $F\subset X$ containing $\span\{\nu(A)\colon A\in\F\}$, a~space $E\in\mathcal{E}$, and an isomorphism $S\colon F\to E$ such that $\n{x}\leq{Sx}\leq\lambda\n{x}$ for every $x\in F$.

Take any number $v>v(\omega,Y)$. Since the function $T_ES\circ\nu\colon\F\to Y_E\subset Y$ is $\lambda d$-additive, there is a~vector measure $\mu\colon\F\to Y$ such that $\n{(T_ES\circ\nu)(A)-\mu(A)}\leq\lambda dv$ for each $A\in\F$. Now, let $\pi_E\colon Y\to Y_E$ be a~projection with $\n{\pi_E}\leq c$. Then $\pi_E\circ\mu\colon\F\to Y_E$ is a~vector measure such that $\n{(T_ES\circ\nu)(A)-(\pi_E\circ\mu)(A)}\leq\lambda vcd$ for each $A\in\F$. Finally, $S^{-1}T^{-1}\pi_E\circ\mu\colon\F\to F\subset X$ is also a~vector measure, and since $\n{S^{-1}}\leq 1$ and $\n{T^{-1}}\leq 1$, it satisfies $$\n{\nu(A)-(S^{-1}T^{-1}\pi_E\circ\mu)(A)}\leq\lambda cdv\quad\mbox{for each }A\in\F.$$ This shows that $X$ has the $\omega$-$\SVM$ property with $v(\omega,X)\leq\lambda cdv(\omega,Y)$.
\end{proof}

\begin{corollary}\label{linfty}
If $X$ is a $\mathscr{L}_\infty$-space then $\tau(X)>\omega$.
\end{corollary}
\begin{proof}
Since $X$ is locally $\{\ell_\infty^n\}_{n=1}^\infty$, the conclusion follows from Proposition \ref{locally} applied for $Y=\ell_\infty$.
\end{proof}

In particular, for every compact Hausdorff space $\Omega$ we have $\tau(\mathcal{C}(\Omega))>\omega$ as all the $\mathcal{C}(\Omega)$-spaces are $\mathscr{L}_\infty$-spaces, which follows from the partition of unity of compact Hausdorff spaces (see, {\it e.g.}, the remarks following \cite[Definition II.5.2]{lindenstrauss_tzafriri}). Therefore, Corollary \ref{linfty} yields a~generalisation of Proposition \ref{omega_c0} with a~less direct proof.
%%%%%%%%%%%%%%%%%%%%%%%%%%%%%%%%%%%%%%%%%%%
%%%%%%%%%%%%%%%%%%%%%%%%%%%%%%%%%%%%%%%%%%%
\section{Background on the three-space problem}
\noindent
This section has preparatory character. We will recall some definitions and facts from the $\mathsf{3SP}$ theory which are essential for our purposes. For more information the reader is referred to Castillo and Gonz\'ales \cite{castillo_gonzalez} and  Cabello S\'anchez and Castillo \cite{sanchez_castillo (dissertationes)}.

Let $X$, $Y$, $Z$ be Banach spaces. A short exact sequence is a diagram 
\begin{equation}\label{ex1}
\exi{i}{q}{Y}{Z}{X}
\end{equation}
where each arrow represents some operator and the image of each arrow coincides with the kernel of the following one. Therefore, $i$ is an injection and $q$ is a quotient operator. Moreover, $Y$ is isomorphic to a subspace of $Z$ and the Open Mapping Theorem ensures that $Z/Y\simeq X$ via an isomorphism $T\colon Z/Y\to X$ with $T\circ\pi=q$ ($\pi$ being the canonical map from $Z$ onto $Z/Y$). We say that two exact sequences of Banach spaces $\ex{Y}{Z_1}{X}$ and $\ex{Y}{Z_2}{X}$ are {\it equivalent} if there exists an operator $T\colon Z_1\to Z_2$ such that the diagram 
$$
\xymatrix{
0\ar[r] & Y\ar[r] \ar@{=}[d] & Z_1\ar[r] \ar[d]^T & X\ar[r] \ar@{=}[d] & 0\\ 
0\ar[r] & Y\ar[r] & Z_2\ar[r] & X\ar[r] & 0
}
$$
is commutative; by the `$3$-lemma' (\cite[\S 1.1]{castillo_gonzalez}) and the Open Mapping Theorem, $T$ must be then an isomorphism. For any two Banach spaces $X$ and $Y$ we have always the trivial exact sequence 
\begin{equation}\label{ex2}
\ex{Y}{Y\oplus X}{X}
\end{equation}
produced by the direct sum, jointly with the natural embedding and projection. We say that exact sequence \eqref{ex1} {\it splits} if and only if it is equivalent to \eqref{ex2}. Of course, in such a case we must have $Z\simeq X\oplus Y$. However, the converse implication is resoundingly false; one may even find a Banach space $X$ for which there is a non-splitting exact sequence of the form $\ex{X}{X\oplus X}{X}$. This topic is connected to Harte's problem (see \cite[\S 1.10]{castillo_gonzalez} for further discussion). The following well-known result gives a deeper insight into the splitting property of exact sequences (for the proof consult \cite[\S 1.1]{castillo_gonzalez}).
\begin{proposition}\label{prop_21}
Assume {\rm\eqref{ex1}} is an exact sequence of Banach spaces $X$, $Y$ and $Z$. Then the following conditions are equivalent each to the other:
\begin{itemize*}
\item[{\rm (i)}] exact sequence {\rm\eqref{ex1}} splits;
\item[{\rm (ii)}] the inclusion $i$ admits a retraction, i.e. an operator $r\colon Z\to Y$ with $r\circ i=\id_Y$;
\item[{\rm (iii)}] the quotient $q$ admits a section {\rm (}lifting{\rm )}, i.e. an operator $s\colon X\to Z$ with $q\circ s=\id_X$;
\item[{\rm (iv)}] the subspace $i(Y)$ is complemented in $Z$.
\end{itemize*} 
\end{proposition}
By a {\it twisted sum} of Banach spaces $Y$ and $X$ we mean any Banach space $Z$ with a subspace isomorphic to $Y$ so that $Z/Y$ is isomorphic to $X$ (the order is important), that is, a~space $Z$ for which there is an exact sequence of the form \eqref{ex1}. By saying that two twisted sums of $X$ and $Y$ are {\it equivalent} we mean that the exact sequences they generate are equivalent in the former sense. The symbol $\Ext$ stands for the functor which assigns to every couple of Banach spaces all the equivalence classes of twisted sums of these spaces. In particular, $\Ext(X,Y)=0$ if and only if every twisted sum of $Y$ and $X$ is trivial, {\it i.e.} it is equivalent to $Y\oplus X$.

In the more general setting, for $F$-spaces, we may similarly consider the notions of: exact sequence, equivalence, splitting and twisted sum. We say that a pair $(X,Y)$ of two $F$-spaces {\it splits} (cf. \cite{kalton}) if whenever $Z$ is an $F$-space containing $Y$ such that $Z/Y\simeq X$, the exact sequence $\ex{Y}{Z}{X}$ splits. This is equivalent to the fact that for any $F$-space $Z$ every operator $T\colon X\to Z/Y$ admits a lifting, {\it i.e.} an operator $\w{T}\colon X\to Z$ such that $\pi\circ\w{T}=T$, where $\pi\colon Z\to Z/Y$ is the quotient map (see \cite[Theorem 3.1]{kalton}). Let us stress that for Banach spaces $X$ and $Y$ the condition that $(X,Y)$ splits is stronger than $\Ext(X,Y)=0$, since in the latter one we claim the splitting property only for these middle spaces $Z$ which are locally convex.
\begin{definition}[cf. \cite{kalton}, \cite{sanchez_castillo (duality)}, \cite{castillo}, \cite{castillo_gonzalez}]\label{quasi_linear}
Let $X$ and $Y$ be quasi-normed spaces. A mapping $f\colon X\to Y$ is called {\it quasi-linear} if it satisfies the two conditions:
\begin{itemize*}
\item[{\rm (i)}] $f(\lambda x)=\lambda f(x)$ for $\lambda\in\R$, $x\in X$;
\item[{\rm (ii)}] $\n{f(x+y)-f(x)-f(y)}\leq c(\n{x}+\n{y})$ for $x,y\in X$,
\end{itemize*}
where $c<\infty$ is a constant independent on $x$ and $y$. We denote by $\Delta(f)$ the least possible constant in condition (ii) and by $\Lambda(X,Y)$ the space of all quasi-linear maps from $X$ into~$Y$.

A mapping $f\colon X\to Y$ is called {\it zero-linear} if it satisfies condition (i) and also
\begin{itemize*}
\item[{\rm (iii)}] $\Bigl\| f\bigl(\sum_{i=1}^nx_i\bigr)-\sum_{i=1}^nf(x_i)\Bigr\|\leq C\sum_{i=1}^n\n{x_i}$ for $x_1,\ldots ,x_n\in X$,
\end{itemize*}
where $C<\infty$ is a constant independent on any choice of elements $x_1,\ldots ,x_n\in X$. We denote by $Z(f)$ the least possible constant in condition (iii) and by $\Xi(X,Y)$ the space of all zero-linear maps from $X$ into $Y$.
\end{definition}

Now, we proceed to connections between stability properties of quasi-linear maps and splitting properties of exact sequences. 
\begin{theorem}[Kalton \cite{kalton}]\label{st_Kalton}
Let $X$ and $Y$ be quasi-Banach spaces and let $X_0$ be a dense subspace of $X$. Then the following conditions are equivalent:
\begin{itemize*}
\item[{\rm (i)}] the pair $(X,Y)$ splits;
\item[{\rm (ii)}] if $f\in\Lambda(X_0,Y)$ then there exists a linear map $h\colon X_0\to Y$ and $L<\infty$ such that $$\n{f(x)-h(x)}\leq L\n{x}\quad\mbox{for }x\in X_0;$$
\item[{\rm (iii)}] there is a constant $B<\infty$ {\rm (}depending only on $X_0$ and $Y${\rm )} such that for every $f\in\Lambda(X_0,Y)$ there exists a linear map $h\colon X_0\to Y$ with $$\n{f(x)-h(x)}\leq B\cdot\Delta(f)\n{x}\quad\mbox{for }x\in X_0.$$
\end{itemize*}
\end{theorem}\noindent
There were Cabello S\'anchez and Castillo \cite{sanchez_castillo (duality)} who adapted the theory of quasi-linear maps to the locally convex setting. In \cite{sanchez_castillo (uniform)} it is remarked that following the line of Kalton's proof of Theorem \ref{st_Kalton} one may obtain an analogous version for Banach spaces and zero-linear maps; we note it in the form below.
\begin{theorem}[Mainly Kalton \cite{kalton}]\label{stz_Kalton}
Let $X$ and $Y$ be Banach spaces and let $X_0$ be a dense subspace of $X$. Then the following conditions are equivalent:
\begin{itemize*}
\item[{\rm (i)}] $\Ext(X,Y)=0$;
\item[{\rm (ii)}] if $f\in\Xi(X_0,Y)$ then there exists a linear map $h\colon X_0\to Y$ and $L<\infty$ such that $$\n{f(x)-h(x)}\leq L\n{x}\quad\mbox{for }x\in X_0;$$
\item[{\rm (iii)}] there is a constant $B<\infty$ {\rm (}depending only on $X_0$ and $Y${\rm )} such that for every $f\in\Xi(X_0,Y)$ there exists a linear map $h\colon X_0\to Y$ with $$\n{f(x)-h(x)}\leq B \cdot Z(f)\n{x}\quad\mbox{for }x\in X_0.$$
\end{itemize*}
\end{theorem}

Given any locally bounded $F$-spaces $X$, $Y$ and any quasi-linear map $f\colon X\to Y$, we can construct a twisted sum of $Y$ and $X$, which algebraically is just $Y\oplus X$ and is equipped with the quasi-norm given by $\n{(y,x)}=\n{x}+\n{y-f(x)}$. We shall denote such a twisted sum by $Y\oplus_fX$ (cf. \cite{kalton_peck}). In the case where $X$ and $Y$ are Banach spaces and $f$ is zero-linear, $Y\oplus_fX$ is a Banach space being a representative for one of the members of $\Ext(X,Y)$. The stability properties exhibited in Theorems \ref{st_Kalton} and \ref{stz_Kalton} may be described in the language of a {\it distance} between quasi-linear maps (cf. \cite{sanchez_castillo (duality)}) which is defined as follows: for homogeneous maps $f$ and $g$, acting between the same quasi-normed spaces $X$ and $Y$, put $$\dist(f,g)=\sup\{\n{f(x)-g(x)}:\,\n{x}\leq 1\}$$(of course, it may happen that it equals $\infty$). In this language, two exact sequences $\ex{Y}{Y\oplus_fX}{X}$ and $\ex{Y}{Y\oplus_gX}{X}$ are equivalent if and only if the difference $f-g$ is trivial in the sense that for some linear mapping $h\colon X\to Y$ we have $\dist(f-g-h)<\infty$. In particular, the sequence $\ex{Y}{Y\oplus_fX}{X}$ splits if and only if $f$ is at finite distance from some linear mapping.

The following result is, by a clear mile, one of the most important results in the whole theory of the $\mathsf{3SP}$ problem:
\begin{theorem}[Kalton \& Peck \cite{kalton_peck}]\label{KP}
Let $X$ and $Y$ be locally bounded $F$-spaces. Then for every twisted sum $Z$ of $Y$ and $X$ there is an $f\in\Lambda(X,Y)$ such that the sequence $\ex{Y}{Z}{X}$ is equivalent to $\ex{Y}{Y\oplus_fX}{X}$.
\end{theorem}
A corresponding result in the locally convex setting reads as follows.
\begin{theorem}[Cabello S\'anchez \& Castillo \cite{sanchez_castillo (uniform)}]\label{KPz}
Let $X$ and $Y$ be Banach spaces. Then for every locally convex twisted sum $Z$ of $Y$ and $X$ there is an $f\in\Xi(X,Y)$ such that that the sequence $\ex{Y}{Z}{X}$ is equivalent to $\ex{Y}{Y\oplus_fX}{X}$.
\end{theorem}

Although quasi-linear and zero-linear maps have nothing in common with continuity, it is (surprisingly) possible to extend them from dense subspaces with preserving quasi- or zero-linearity. The following result is a consequence of Theorem \ref{KP}.
\begin{theorem}[Kalton \& Peck \cite{kalton_peck}]\label{extension}
Let $X$ and $Y$ be locally bounded $F$-spaces and $X_0$ a dense subspace of $X$. For every $f_0\in\Lambda(X_0,Y)$ there exists an extension $f\in\Lambda(X,Y)$ of $f_0$. Moreover, for every extension $g\in\Lambda(X,Y)$ of $f_0$ we have $\dist(f-g-h)<\infty$ for some linear map $h\colon X\to Y$.
\end{theorem}

The zero-linear analogue may proved very much like the above theorem (by using Theorem \ref{KPz} instead of \ref{KP}), but even something more may be stated. The result below holds true also in a wider context of normed groups (cf. \cite{sanchez_castillo (dissertationes)}).
\begin{theorem}[Cabello S\'anchez \& Castillo \cite{sanchez_castillo (dissertationes)}]\label{Zextension}
Let $X$ and $Y$ be Banach spaces and $X_0$ a dense subspace of $X$. If $f_0\in\Xi(X_0,Y)$, then every quasi-linear extension $X\to Y$ of $f_0$ is zero-linear. 
\end{theorem}
%%%%%%%%%%%%%%%%%%%%%%%%%%%%%%%%%%%%%%%%%%%
%%%%%%%%%%%%%%%%%%%%%%%%%%%%%%%%%%%%%%%%%%%
\section{Necessary conditions for the $\kappa$-$\SVM$ and the $\SVM$ properties}
\noindent
First, we give a condition which is necessary for a~Banach space, complemented in its bidual, to have the $\omega$-$\SVM$ property. The proof repeats the argument of the proof of \cite[Theorem 6.3]{kalton_roberts}. Recall that a~Banach space $Y$ is a~$\mathscr{L}_\infty$-{\it space} if there exists a~constant $c\geq 1$ such that every finite-dimensional subspace of $Y$ is contained in another finite-dimensional subspace $F$ of $Y$ such that $d_\mathsf{BM}(F,\ell_\infty^m)\leq c$, where $m=\dim F$.
\begin{theorem}\label{T1}
If $X$ is a Banach space complemented in its bidual, which has the $\omega$-{\rm SVM} property, then for every Banach space $Y$, which is a~$\mathscr{L}_\infty$-space, the pair $(Y,X)$ splits.
\end{theorem}
\begin{proof}
First, we shall prove that for any finite set $\Om$, any fixed $v>v(\om,X)$, and any quasi-linear map $h\colon\ell_\infty(\Om)\to X$, with $\delta=\Delta(h)$, there exists a~linear map $H\colon\ell_\infty(\Om)\to X$ such that
\begin{equation}\label{T11}
\n{h(t)-H(t)}\leq 4\delta(2+v)\n{t}\quad\mbox{for }t\in\ell_\infty(\Om).
\end{equation}
To this end observe that the function $\nu\colon 2^\Om\to X$ given by $\nu(A)=h(\ind_A)$ satisfies $$\n{\nu(A\cup B)-\nu(A)-\nu(B)}\leq 2\delta\quad\mbox{for }A,B\in 2^\Om, A\cap B=\varnothing .$$Hence, for any fixed $v>v(\om,X)$, there is a~vector measure $\mu\colon 2^\Om\to X$ such that $$\n{h(\ind_A)-\mu(A)}\leq 2\delta v\quad\mbox{for }A\in 2^\Om .$$Let $H\colon\ell_\infty(\Om)\to X$ be the natural linear extension of $\mu$. Then $r=h-H$ is quasi-linear and $\n{r(\ind_A)}\leq 2\delta v$ for each $A\in 2^\Om$. We will show that $\n{r(t)}\leq 4\delta(2+v)\n{t}$ for $t\in\ell_\infty(\Om)$, which gives \eqref{T11}.

If $t\in\ell_\infty(\Om)$, then $t=\sum_{\alpha\in\Om}t(\alpha)\mathrm{e}_\alpha$, where $\mathrm{e}_\alpha=\ind_{\{\alpha\}}$. Using the inequality that defines quasi-linear maps $(\abs{\Om}-1)$ times, and using the fact that $\n{\sum_{\alpha\in A}t(\alpha)\mathrm{e}_\alpha}\leq\n{t}$ for each $A\subset\Om$, we obtain $$\Biggl\|r(t)-\sum_{\alpha\in\Om}t(\alpha)r(\mathrm{e}_\alpha)\Biggr\|\leq 2(\abs{\Om}-1)\delta\n{t},$$thus $$\n{r(t)}\leq 2\bigl(\abs{\Om}+v\abs{\Om}-1\bigr)\delta\n{t}.$$

Now, for $0\leq t\leq\ind_\Om$ and each $m\in\N$ we may find $A_1,\ldots ,A_m\in 2^\Om$ such that $$\Biggl\|t-\sum_{k=1}^m{1\over 2^k}\ind_{A_k}\Biggr\|\leq 2^{-m}.$$Moreover, using quasi-linearity of $r$ recursively, we get
$$\Biggl\|r\Biggl(\sum_{k=1}^m{1\over 2^k}\ind_{A_k}\Biggr)-\sum_{k=1}^m{1\over 2^k}r(\ind_{A_k})\Biggr\|\leq\delta\sum_{k=1}^m{k\over 2^k}\leq 2\delta ,$$hence $$\Biggl\|r\Biggl(\sum_{k=1}^m{1\over 2^k}\ind_{A_k}\Biggr)\Biggr\|\leq 2\delta(1+v).$$
Consequently, $$\n{r(t)}\leq 2\delta(1+v)+2^{-m+1}\delta\bigl(\abs{\Om}+v\abs{\Om}-1\bigr)+\delta(1+2^{-m})$$and, letting $m\to\infty$, we conclude that $$\n{r(t)}\leq  \delta(3+2v).$$

Finally, for every $t\in\ell_\infty(\Om)$ with $t=t^+-t^-$, where $t^+,t^-\geq 0$ and $\n{t^+}, \n{t^-}\leq\n{t}$, we have 
$$\n{r(t)}\leq\delta(\n{t^+}+\n{t^-})+\n{r(t^+)}+\n{r(t^-)}\leq 4\delta(2+v)\n{t},$$
which yields \eqref{T11}.

Now, let $Y$ be a $\mathscr{L}_\infty$-space and $f\colon Y\to X$ be a~quasi-linear map with $\Delta=\Delta(f)$. Then there is a constant $c\geq 1$ such that for every finite-dimensional subspace $F\subset Y$ there is a~finite-dimensional subspace $G\supset F$ of $Y$ and a~linear isomorphism $T\colon G\to\ell_\infty^m$ with $\n{T}\cdot\n{T^{-1}}\leq c$.

Let $F\subset Y$ be an arbitrary finite-dimensional space and choose $G$ as above. Then $f\circ T^{-1}$ is quasi-linear on $\ell_\infty^m$ with $\Delta(f\circ T^{-1})\leq\n{T^{-1}}\Delta$. By the first part of the proof, for any $v>v(\om,X)$, there is a linear mapping $H\colon\ell_\infty^m\to X$ satisfying $$\n{(f\circ T^{-1})(t)-H(t)}\leq 4\Delta (2+v)\n{T^{-1}}\n{t}\quad\mbox{for }t\in\ell_\infty^m.$$Define $H_F\colon Y\to X$ by $H_F(y)=H(Ty)$ for $y\in F$ and $H_F(y)=0$ for $y\in Y\setminus F$. 

Let $\Ra$ be the family of all finite-dimensional subspaces of $Y$. Then $(\Ra,\subset)$ is a~directed set and $(H_F)_{F\in\Ra}$ is a~net of functions such that for each $y\in Y$ we have $$\n{H_F(y)}\leq 4\Delta (2+v)\n{T}\n{T^{-1}}\n{y}+\n{f(y)},$$where $T$ is the linear isomorphism matched with the subspace $F$ as above. Therefore, $$\n{H_F(y)}\leq 4c\Delta (2+v)\n{y}+\n{f(y)}\quad\mbox{for }y\in Y,$$that is, for each $y\in Y$ the set $\{ H_F(y):\, F\in\Ra\}$ is embedded into the $w^\ast$-compact unit ball of $X^{\ast\ast}$. Hence, we may choose a subnet of $(H_F)_{F\in\Ra}$ which is pointwise convergent (in the \ws\ topology of $X^{\ast\ast}$) to a~linear map $h\colon Y\to X^{\ast\ast}$. Clearly, $$\n{f(y)-h(y)}\leq 4c\Delta (2+v)\n{y}\quad\mbox{for }y\in Y.$$Now, let $P\colon X^{\ast\ast}\to X$ be a~bounded projection. Then, since $Pf(y)=f(y)$ for each $y\in Y$, we have $$\n{f(y)-(P\circ h)(y)}\leq 4c\Delta (2+v)\n{P}\n{y}\quad\mbox{for }y\in Y,$$and $P\circ h\colon Y\to X$ is a linear map. In view of Theorem \ref{st_Kalton}, the proof is completed.
\end{proof}

The next result gives another necessary condition which is a weaker variation of the preceding one but omits the complementability assumption. In the first assertion the space $\ell_\infty(\Gamma)$ plays a~role of a~testing space. We shall see in Section 6 that there is a~whole class of such testing spaces, which includes $\ell_\infty(\Gamma)$, and the idea of the subsequent proof will be used there once again.
\begin{theorem}\label{T2}
Let $\Gamma$ be a cardinal number and $X$ be a~Banach space. Then:
\begin{itemize*}
\item[{\rm (i)}] if $X$ has the $\bigl({2^\Gamma}\bigr)^+$-$\SVM$ property then the pair $(\ell_\infty(\Gamma),X)$ splits;
\item[{\rm (ii)}] if $X$ has the~$\Gamma^+$-$\SVM$ property then the pair $(c_0(\Gamma),X)$ splits.
\end{itemize*}
\end{theorem}
\begin{proof}
%First, we will prove that $(\ell_\infty(\Gamma),X)$ splits, whenever $X$ has the $\bigl(2^\Gamma\bigr)^+$-SVM property. 
(i): Fix any number $v>v\bigl(\bigl(2^\Gamma\bigr)^+,X\bigr)$ and let $Y_0\subset\ell_\infty(\Gamma)$ be the (dense) subspace consisting of all step functions. We shall show that for each quasi-linear map $f\colon Y_0\to X$, with $\delta=\Delta(f)$, there is a linear map $h\colon Y_0\to X$ satisfying
\begin{equation}\label{T21}
\n{f(t)-h(t)}\leq 4\delta(2+v)\n{t}\quad\mbox{for }t\in Y_0.
\end{equation}
Arguing as above, we infer that there is a vector measure $\mu\colon 2^\Gamma\to X$ such that $$\n{f(\ind_A)-\mu(A)}\leq 2\delta v\quad\mbox{for }A\in 2^\Gamma .$$Let $h\colon Y_0\to X$ be the natural linear extension of $\mu$ to the space of all step functions. Then $r=f-h$ is quasi-linear and $\n{r(\ind_A)}\leq 2\delta v$ for each $A\in 2^\Gamma$.

Every $t\in Y_0$ has the form $t=\sum_{j=1}^mt_j\ind_{A_j}$ for some $m\in\N$, $t_j\in\R$ and pairwise disjoint sets  $A_j\subset\Gamma$. Therefore, $$\n{r(t)}\leq 2\bigl(m+vm-1\bigr)\delta\n{t}.$$

For $t\in Y_0$, $0\leq t\leq\ind_{\Gamma}$, and each $n\in\N$ we may find $B_1,\ldots ,B_n\in 2^\Gamma$ such that $$\Biggl\|t-\sum_{k=1}^n{1\over 2^k}\ind_{B_k}\Biggr\|\leq 2^{-n}.$$As in the preceding proof, we may show that $$\n{r(t)}\leq 2\delta(1+v)+2^{-n+1}\bigl(m+vm-1\bigr)+\delta(1+2^{-n}).$$Letting $n\to\infty$, we conclude that $\n{r(t)}\leq  \delta(3+2v(X))$, and by splitting an~arbitrary $t\in Y_0$ into its positive and negative parts we get inequality \eqref{T21}.

\vspace*{2mm}
(ii): Now, fix any number $v>v(\Gamma^+,X)$ and let $c_{00}(\Gamma)$ be the (dense) subspace of $c_0(\Gamma)$ consisting of all finitely supported sequences. Let also $f\colon c_{00}(\Gamma)\to X$ be a quasi-linear map with $\delta=\Delta(f)$. We shall show that there exists a linear map $h\colon c_{00}(\Gamma)\to X$ satisfying
\begin{equation}\label{T22}
\n{f(t)-h(t)}\leq 4\delta(2+v)\n{t}\quad\mbox{for }t\in c_{00}(\Gamma).
\end{equation}

Let $\F$ be the algebra of all subsets of $\Gamma$ which are either finite, or have finite complements. Define $\nu\colon\F\to X$ by $$\nu(A)=\left\{\begin{array}{rl}f(\ind_A) & \mbox{if }A\mbox{ is finite},\\ -f(\ind_{\Gamma\setminus A}) & \mbox{if }\Gamma\setminus A\mbox{ is finite}.\end{array}\right.$$An easy verification shows that $$\n{\nu(A\cup B)-\nu(A)-\nu(B)}\leq 2\delta\quad\mbox{for }A,B\in\F ,\,\, A\cap B=\varnothing .$$Hence, there is a vector measure $\mu\colon\F\to X$ such that $$\n{f(\ind_A)-\mu(A)}\leq 2\delta v\quad\mbox{for any finite }A\in 2^\Gamma .$$Define $h\colon c_{00}(\Gamma)\to X$ as the linear map induced in a natural way from the measure $\mu$. Then, by the same argument as before, inequality \eqref{T22} is valid and again appealing to Theorem \ref{st_Kalton} completes the proof.
\end{proof}

Before deriving some corollaries from the just-proved theorems let us recall the following result, \cite[Corollary 7.3]{aviles}, which we will find useful also in Section 8.
\begin{theorem}[Avil\'es, Cabello S\'anchez, Castillo, Gonz\'alez \& Moreno \cite{aviles}]\label{pulgarin_thm}
Let $X$ be a separable Banach space and $Y$ be a~Banach space. Suppose that $X$ has the bounded approximation property or $Y$ has the uniform approximation property. Then $\mathrm{Ext}(X,Y)=0$ implies $\mathrm{Ext}(Y^\ast,X^\ast)=0$. Consequently, 
\begin{equation}\label{l1c0}
\Ext(X^\ast,\ell_1)\not=0\quad\mbox{implies}\quad\Ext(c_0,X)\not=0.
\end{equation}
\end{theorem}

\begin{corollary}\label{Cor1}
If $X$ is a Banach space containing $\{\ell_p^n\}_{n=1}^\infty$ uniformly complemented, for some $1\leq p<\infty$, then $\tau(X)=\omega$. Consequently, for every $\mathscr{L}_p$-space $X$, with $1\leq p<\infty$, we have $\tau(X)=\omega$.
\end{corollary}
\begin{proof}
In light of Propositions \ref{comp} and \ref{locally}, and the fact that every $\mathscr{L}_p$-space contains a~complemented copy of $\ell_p$ (cf. \cite[Proposition II.5.5]{lindenstrauss_tzafriri}) it is enough to show that $\tau(\ell_p)=\omega$ whenever $1\leq p<\infty$.

By a result of Cabello S\'anchez and Castillo (\cite[Example 4.3]{sanchez_castillo (uniform)}), we have $\Ext(c_0,\ell_1)\not=0$, whence Theorem~\ref{T2} implies $\tau(\ell_1)\leq\omega_1$. In the same paper (see \cite[Example~4.1]{sanchez_castillo (uniform)}) it was shown that $\Ext(\ell_2,\ell_1)\not=0$. Since for every $1<p<\infty$ the space $\ell_p$ contains $\{\ell_2^n\}_{n=1}^\infty$ uniformly complemented, the last statement easily implies that $\Ext(\ell_p,\ell_1)\not=0$, thus \eqref{l1c0} gives $\Ext(c_0,\ell_p)\not=0$ and, by Theorem~\ref{T2}, we have $\tau(\ell_p)\leq\omega_1$ for $1<p<\infty$. Consequently, the dichotomy in Proposition \ref{bidual} implies that $\tau(\ell_p)=\omega$ for $1\leq p<\infty$.
\end{proof}

With this tool in hand, we are ready to show that the $\SVM$ character equals $\omega$ for many classical Banach spaces. For completeness, let us recall briefly some definitions.

The $p$th {\it James space} $\mathcal{J}_p$ (for $1\leq p<\infty$) is defined to be the Banach space of all (real) sequences $\xi=(\xi(n))_{n=1}^\infty$ such that $$\n{\xi}_{\mathcal{J}_p}:=\sup\Biggl\{\Bigl(\sum_{j=1}^n\abs{\xi(k_j)-\xi(k_{j-1})}^p\Bigr)^{1/p}\colon 1\leq k_0<k_1<\ldots <k_n\Biggr\}<\infty,$$equipped with the so-defined norm $\n{\cdot}_{\mathcal{J}_p}$. The space $\mathcal{J}_p$ contains a~complemented copy of $\ell_p$; such copies may be even found in every infinite-dimensional subspace of $\mathcal{J}_p$, by a~result of Casazza, Lin and Lohman \cite[Corollary 11]{casazza_lin_lohman}. For further information, see \cite[Chapter 2]{forest}, \cite[\S 3.4]{albiac_kalton}, and the references therein.

The $p$th {\it Johnson--Lindenstrauss space} $\JL_p$ (for $1\leq p<\infty$) is defined as follows: Let $\{N_\gamma\}_{\gamma\in\Gamma}$ be an~almost disjoint family ({\it i.e.} $\abs{N_{\gamma_1}\cap N_{\gamma_2}}<\omega$ for all $\gamma_1\not=\gamma_2$ from $\Gamma$) with cardinality of the continuum and consisting of infinite subsets of $\omega$. Now, consider the space $V=\span(c_0\cup\{\ind_{N_\gamma}\colon\gamma\in\Gamma\})$ generated by all the sequences tending to zero and the characteristic functions of all $N_\gamma$'s. Every element $x$ of $V$ is of the form 
\begin{equation}\label{xJL}
x=y+\sum_{j=1}^ka_{\gamma_j}\ind_{N_{\gamma_j}},\quad\mbox{where }\, y\in c_0,\,\,\, a_{\gamma_j}\in\R,\,\,\,\gamma_j\in\Gamma\,\mbox{ and }\,\gamma_i\not=\gamma_j\,\mbox{ for }\, i\not=j.\end{equation}
Moreover, the scalars $a_{\gamma_j}$ ($1\leq j\leq k$) are uniquely determined by $x$, so the formula 
$$
\n{x}_{\JL_p}:=\max\Biggl\{\n{x}_\infty,\Bigl(\sum_{j=1}^k\abs{a_{\gamma_j}}^p\Bigr)^{1/p}\Biggr\}
$$
is well-posed and defines a~norm on $V$. The space $\JL_p$ is then the completion of $V$ with respect to this norm (for more details, see \cite{johnson_lindenstrauss}). If $p=\infty$ the definition is algebraically the same, but we take the $\ell_\infty$-norm. The space $\JL_\infty$ will be dealt with in Section 6; now we focus on the case $1\leq p<\infty$.

Observe that $\JL_p$ contains $\{\ell_p^n\}_{n=1}^\infty$ uniformly complemented. To see this, let $n\in\N$ and pick $n$ distinct $\gamma_1,\ldots ,\gamma_n\in\Gamma$. Define $F_n=\bigcup_{1\leq i<j\leq n}(N_{\gamma_i}\cap N_{\gamma_j})$ which is a~finite set. Then the sets $N_{\gamma_j}\setminus F_n$, for $1\leq j\leq n$, are pairwise disjoint. Consider the subspace $W_n$ of $\JL_p$ consisting of all the sequences from $\JL_p$ supported on the set $S_n:=\bigcup_{j=1}^nN_{\gamma_j}\setminus F_n$ and constant on each of the sets $N_{\gamma_j}\setminus F_n$. Plainly, $W_n$ is then isometrically isomorphic to $\ell_p^n$, and it is easily seen that it is also $1$-complemented in $\JL_p$. For let $\pi_n\colon V\to W_n$ be a~map which for each $x\in V$, being of the form \eqref{xJL} (with $k\geq n$ and some of $a_j$'s possibly equal to zero), cancels the term $y\in c_0$, as well as all the scalars $a_{\gamma_j}$ with $j>n$, and projects the resulting sequence onto $S_n$, that is,
$$
\pi_n\Bigl(y+\sum_{j=1}^ka_{\gamma_j}\ind_{N_{\gamma_j}}\Bigr)=\sum_{j=1}^na_{\gamma_j}\ind_{N_{\gamma_j}\cap S_n}.
$$
Then $\pi_n$ is a~norm-one projection onto $W_n$, so it admits an~extension to a~norm-one projection $\JL_p\to W_n$.

The $p$th {\it Schreier space} $\mathcal{S}_p$ (for $1\leq p<\infty$) is defined to be the Banach space of all (real) sequences $\xi=(\xi(n))_{n=1}^\infty$ such that
$$
\n{\xi}_{\mathcal{S}_p}:=\sup\Biggl\{\Bigl(\sum_{j=1}^n\abs{\xi(k_j)}^p\Bigr)^{1/p}\colon n,k_1,\ldots ,k_n\in\N\mbox{ and }n\leq k_1<\ldots <k_n\Biggr\}<\infty,
$$
equipped with the so-defined norm $\n{\cdot}_{\mathcal{S}_p}$. Plainly, for any $n\in\N$ the subspace of $\mathcal{S}_p$ consisting of all sequences supported on the interval $\{n,n+1,\ldots ,2n-1\}$ is isometrically isomorphic to $\ell_p^n$ and $1$-complemented, which shows that $\mathcal{S}_p$ contains $\{\ell_p^n\}_{n=1}^\infty$ uniformly complemented.

The {\it James tree space} $\mathcal{JT}$ is defined to be the Banach space of all (real) sequences $\xi=(\xi(t))_{t\in\mathcal{T}}$, indexed by the elements of the dyadic tree $\mathcal{T}=\{(n,j)\colon n\in\N, 0\leq j<2^n\}$ (ordered by letting $(m,i)<(n,j)$ iff $m<n$), such that
$$
\n{\xi}_{\mathcal{JT}}:=\sup\Biggl\{
\Bigl(\sum_{j=1}^n\Bigl|\sum_{i\in I_j}\xi(i)\Bigr|^2\Bigr)^{1/2}\colon 
\begin{array}{c}
n\in\N\mbox{ and }I_1,\ldots ,I_n\mbox{ are pairwise}\\
\mbox{disjoint segments of }\mathcal{T}\end{array}
\Biggr\}<\infty
$$
(by a {\it segment} we mean a~subset of $\mathcal{T}$ which, for some $s,t\in\mathcal{T}$, is the maximal totally ordered set with $s$ as its minimal element and $t$ as its maximal element), equipped with the so-defined norm $\n{\cdot}_{\mathcal{JT}}$. For any fixed {\it branch} $\mathcal{B}$ (maximal totally ordered subset of $\mathcal{T}$), the space of all sequences from $\mathcal{JT}$ supported on $\mathcal{B}$ forms a~$1$-complemented subspace isomorphic to the James space $\mathcal{J}_2$ (cf. \cite[Proposition 3.a.7(b)]{forest}). Hence, $\mathcal{JT}$ contains an~isomorphic copy of $\ell_2$ complemented. According to a~result of Jebreen, Jamjoom and Yost \cite[Lemma 2.4]{jebreen_jamjoom_yost}, the standard predual $\mathcal{JT}_\ast$ of $\mathcal{JT}$ contains $\{\ell_1^n\}_{n=1}^\infty$ uniformly complemented, hence so does the dual $\mathcal{JT}^\ast$. For further information, see \cite[Chapter 3]{forest}, \cite[\S 13.4]{albiac_kalton}, and the references therein.

Many examples of Banach spaces having the $\SVM$ character $\omega$ are provided by the class of asymptotic $\ell_p$ spaces (recall that a~Banach space $X$ with a~normalised basis $(e_n)_{n=1}^\infty$ is {\it asymptotic} $\ell_p$ if there is a~constant $c>0$ so that for every $k\in\N$ there is $N=N(k)$ such that every sequence $(x_j)_{j=1}^k$ of $k$ successive, normalised block vectors of $(e_n)_{n=1}^\infty$ with $N<\supp(x_1)<\ldots<\supp(x_k)$ is $c$-equivalent to the canonical basis of $\ell_p^n$), because all such spaces contain $\{\ell_p^n\}_{n=1}^\infty$ uniformly complemented. Let us just mention the famous {\it Tsirelson space} $\mathcal{T}_\theta$, with parameter $\theta\in (0,1)$, being the completion of $c_{00}$ with respect to the norm $\n{\cdot}_{\mathcal{T}_\theta}$ defined by the implicit equation
$$
\n{\xi}_{\mathcal{T}_\theta}=\max\Biggl\{\n{\xi}_\infty,\sup\theta\sum_{j=1}^k\n{\xi\ind_{E_j}}_{\mathcal{T}_\theta}\Biggr\},
$$
the `sup' taken over all {\it admissible} families $\{E_1,\ldots ,E_k\}$ of finite subsets of $\N$, that is, satisfying $k\leq E_1<\ldots<E_k$. The space $\mathcal{T}_{1/2}$ (defined in this form by Figiel and Johnson \cite{figiel_johnson}) is the dual of a~space originally introduced by Tsirelson. Each of the spaces $\mathcal{T}_\theta$, for $\theta\in (0,1)$, is asymptotic $\ell_1$ (cf. \cite[Chapter 3]{dew}). The spaces $\mathcal{T}_\eta$ and $\mathcal{T}_\theta$, for $\eta\not=\theta$, are {\it totally incomparable} (neither contains an infinite-dimensional subspace isomorphic to a~subspace of the other); see \cite[Theorem X.a.3]{casazza_shura}. For the comprehensive study of many other Tsirelson-type spaces, consult \cite{argyros_todorcevic} and \cite{dew}.

To close the list of examples note that the space $\mathscr{B}(\mathcal{H})$ of all operators acting from an infinite-dimensional Hilbert space $\mathcal{H}$ into itself contains a~complemented subspace isomorphic to $\ell_2$. One of such is the subspace of all `row operators' $T\in\mathscr{B}(\mathcal{H})$ of the form $Th=\sum_{j=0}^\infty t_j\langle h,e_j\rangle$, where $(t_j)_{j=0}^\infty$ may be an~arbitrary sequence from $\ell_2$ and $(e_j)_{j<\dim\mathcal{H}}$ is any fixed orthonormal basis of $\mathcal{H}$.

\begin{corollary}\label{Cor2}
Let $X$ be one of the Banach spaces: $\mathcal{J}_p$, $\JL_p$, $\mathcal{S}_p$ {\rm (}for $1\leq p<\infty${\rm )}, $\mathcal{JT}$, $\mathcal{JT}_\ast$, $\mathcal{JT}^\ast$, $\mathcal{T}_\theta$ {\rm (}for $0<\theta<1${\rm )}, or $\mathscr{B}(\mathcal{H})$ {\rm (}for an infinite-dimensional Hilbert space $\mathcal{H}${\rm )}. Then $\tau(X)=\omega$.
\end{corollary}

\begin{corollary}\label{Cor3}
For every compact metric space $\Omega$, except the case where $\mathcal{C}(\Omega)\simeq c_0$ or $\mathcal{C}(\Omega)$ is finite-dimensional, we have $\tau(\mathcal{C}(\Omega))=\omega_1$.
\end{corollary}
\begin{proof}
In light of Proposition \ref{omega_c0}(ii), we shall prove that none of the $\mathcal{C}(\Omega)$-spaces as above has the $\omega_1$-$\SVM$ property. In view of the theorems of Miljutin and Bessaga--Pe\l czy\'nski, it is enough to consider the spaces $\mathcal{C}[0,\omega^{\omega^\alpha}]$, where $\alpha\geq 1$ is a~countable ordinal, and the space $\mathcal{C}[0,1]$.  

By a theorem of Cabello S\'anchez, Castillo, Kalton and Yost (\cite[Theorems 4.1 \&\ 3.5]{sanchez_castillo_kalton_yost}), we have $\Ext(c_0,\mathcal{C}[0,\omega^\omega])\not=0$. Hence, Theorem \ref{T2}(ii) implies that $\mathcal{C}[0,\omega^\omega]$ does not have the $\omega_1$-$\SVM$ property, so neither does any $\mathcal{C}[0,\omega^{\omega^\alpha}]$ (for $\alpha\geq 1$ being a~countable ordinal), since every such space contains $\mathcal{C}[0,\omega^\omega]$ complemented (notice that $[0,\omega^\omega]$ is a~clopen subset of $[0,\omega^{\omega^\alpha}]$ for $\alpha\geq 1$).

By a result of Foia\c{s} and Singer \cite{foias_singer}, we have $\Ext(c_0,\mathcal{C}[0,1])\not=0$ (see also the remarks before Proposition \ref{ext_JL} below). Thus, another appeal to Theorem \ref{T2}(ii) completes the proof.
\end{proof}
%%%%%%%%%%%%%%%%%%%%%%%%%%%%%%%%%%%%%%%%%%%
%%%%%%%%%%%%%%%%%%%%%%%%%%%%%%%%%%%%%%%%%%%
\section{$\kappa$-injectivity implies the $\kappa$-$\SVM$ property}
\noindent
The aim of this section is to derive an analogue of Theorem \ref{injective} in which injectivity is replaced by its weaker version, $\kappa$-injectivity.
\begin{definition}\label{kappa}
Let $\kappa$ be a cardinal number. A~Banach space $X$ is called $\kappa$-{\it injective} if for every Banach space $E$, with density character less than $\kappa$, and every subspace $F\subset E$, every operator $t\colon F\to X$ admits an extension to an operator $T\colon E\to X$. If for some $\lambda\geq 1$ there is always such an~extension with $\n{T}\leq\lambda\n{t}$, then we say that $X$ is $(\lambda,\kappa)$-{\it injective}.

If $\kappa=\omega_1$ we say that $X$ is {\it separably injective}.
\end{definition}
By an `$\ell_1$-sum argument', if $\kappa$ has uncountable cofinality, then every $\kappa$-injective Banach space is $(\lambda,\kappa)$-injective for some $\lambda\geq 1$ (cf. \cite{sanchez}). The above definition, as well as the proposition below, comes from \cite{aviles}, where the reader may find a~detailed exposition on the subject of $\kappa$-injectivity.
\begin{proposition}[Avil\'es, Cabello S\'anchez, Castillo, Gonz\'alez \& Moreno \cite{aviles}]\label{dens}
Let $\kappa$ be a cardinal number and $X$ be a~Banach space. Then, the following assertions are equivalent:
\begin{itemize*}
\item[{\rm (a)}] $X$ is $\kappa$-injective;
\item[{\rm (b)}] for every set $\Gamma$ with $\abs{\Gamma}<\kappa$, every operator from a~subspace of $\ell_1(\Gamma)$ into $X$ admits an extension to $\ell_1(\Gamma)$;
\item[{\rm (c)}] for every Banach space $E$, and every its subspace $F$ such that the density character of $E/F$ is less than $\kappa$, every operator $t\colon F\to X$ has an extension to an~operator $T\colon E\to X$;
\item[{\rm (d)}] if $Z$ is a Banach space containing $X$ and the density character of $Z/X$ is less than $\kappa$, then $X$ is complemented in $Z$;
\item[{\rm (e)}] $\Ext(Z,X)=0$ for every Banach space $Z$ with density character less than $\kappa$.
\end{itemize*}
Moreover, if $X$ is $(\lambda,\kappa)$-injective for some $\lambda\geq 1$, then assertion {\rm (c)} holds true with $\n{T}\leq 3\lambda\n{t}$, whereas assertion {\rm (d)} may be strengthen by saying that $X$ is $3\lambda$-complemented.
\end{proposition}

Now, our goal is to prove the following theorem:
\begin{theorem}\label{k_injective}
Let $\kappa$ be a cardinal number with uncountable cofinality. Then, every $\kappa$-injective Banach space $X$ has the $\kappa$-$\SVM$ property, in other words, $\tau(X)>\kappa$. Moreover, if $X$ is $(\lambda,\kappa)$-injective {\rm (}and then the cofinality of $\kappa$ may be arbitrary{\rm )} then 
\begin{equation}\label{kappa_v}
v(\kappa,X)\leq 24\lambda K.
\end{equation}
\end{theorem}

An important role in our argument is played by a certain Banach space depending on a~given set algebra. We define it as follows. Suppose that $\F\subset 2^\Om$ is an infinite algebra of subsets of $\Om$. Let $X_\F$ be the linear subspace of real-valued functions defined on $\Om$ which is generated by the set $\{\ind_A:\, A\in\F\}$ of characteristic functions of all members of $\F$, considered with the pointwise operations. We equip $X_\F$ with the norm defined by $$\n{x}_\F=\inf\left\{\sum_{i=1}^k\abs{\alpha_i}:\, k\geq 0, \alpha_i\in\R , A_i\in\F\mbox{ and }x=\sum_{i=1}^k\alpha_i\ind_{A_i}\right\} .$$It is easy to check that $\n{\cdot}_\F$ is, in fact, a norm on $X_\F$. We denote $\mathcal{X}_\F$ the completion of $X_\F$ with respect to this norm. Obviously, the density character of the Banach space $\mathcal{X}_\F$ equals the cardinality of $\F$ (just consider finite linear combinations of $\ind_A$, for $A\in\F$, having rational coefficients). 

The importance of the space $\mathcal{X}_\F$ consists in the possibility of producing a zero-linear map from a given $1$-additive function acting on $\F$. This procedure is not so straightforward as the reverse one (which we have seen in Theorems \ref{T1} and \ref{T2}).

We will also need the following well-known renorming theorem \cite[Proposition 1]{pelczynski}.
\begin{theorem}[Pe\l czy\'nski \cite{pelczynski}]\label{Pelczynski}
Let $(X,\n{\cdot})$ be a Banach space containing a subspace $Y$ isomorphic to $Y_1$ by an isomorphism $U\colon Y\to Y_1$ satisfying $\n{y}\leq\n{U(y)}\leq c\n{y}$ for $y\in Y$ with some $c\geq 1$. Then there is a norm $\n{\cdot}^\prime$ on $X$ satisfying $\n{x}\leq\n{x}^\prime\leq c\n{x}$ for $x\in X$ and such that the identity mapping $\id\colon (Y,\n{\cdot}^\prime)\to Y_1$ is an isometry.
\end{theorem}

\begin{proof}[Proof of Theorem \ref{k_injective}]
Assume $X$ is a $\kappa$-injective Banach space, $\F$ is a~set algebra of cardinality less than $\kappa$, and $\nu\colon\F\to X$ is a~$1$-additive function. Pick any $\e>0$. Let $f_0\colon X_\F\to m_0(\Gamma)$ be a mapping satisfying the following conditions:
\begin{itemize*}
\item[(a)] $f_0(\ind_A)=\nu(A)$ for $A\in\F$;
\item[(b)] $f_0(\lambda x)=\lambda f_0(x)$ for $\lambda\in\R$, $x\in X_\F$;
\item[(c)] for each $x\in X_\F$ we have $f_0(x)=\sum_{i=1}^k\alpha_i\nu(A_i)$, where $k\geq 0$, $\alpha_i\in\R$, $A_i\in\F$ (for $1\leq i\leq k$) satisfy $x=\sum_{i=1}^k\alpha_i\ind_{A_i}$ and $\sum_{i=1}^k\abs{\alpha_i}\leq (1+\e)\n{x}_\F$.
\end{itemize*}
In order to construct such a map, put $f_0(0)=0$, and for all $x\in X_\F$ with $\n{x}_\F=1$ define the values $f_0(x)$ in such a manner that they satisfy equalities required in (a) and (c). Next, for each pair of the form $(x,-x)$ pick any of its elements, say $-x$, and replace $f_0(-x )$ by $-f_0(x)$. Finally, extend $f_0$ homogeneously to ensure the validity of (b) and, consequently, also of (c).

Now, we will show that for any $x_1,\ldots ,x_n\in X_\F$ we have 
\begin{equation}\label{Zfm}
\Biggl\|f_0\Bigl(\sum_{i=1}^nx_i\Bigr)-\sum_{i=1}^nf_0(x_i)\Biggr\|\leq 2(1+\e)K\sum_{i=1}^n\n{x_i}_\F,
\end{equation}
and hence $f_0\in\Xi(X_\F,X)$ with $Z(f_0)\leq 2(1+\e)K$. 

Let $k_i\geq 0$, $\alpha_{ij}\in\R$ and $A_{ij}\in\F$ satisfy
\begin{equation}\label{xi}
x_i=\sum_{j=1}^{k_i}\alpha_{ij}\ind_{A_{ij}}\quad\mbox{for }1\leq i\leq n,
\end{equation}
with
$$f_0(x_i)=\sum_{j=1}^{k_i}\alpha_{ij}\nu(A_{ij})\quad\mbox{and}\quad\sum_{j=1}^{k_i}\abs{\alpha_{ij}}\leq (1+\e)\n{x_i}_\F\quad\mbox{for }1\leq i\leq n.$$Similarly, let $\ell\geq 0$, $\beta_j\in\R$ and $B_j\in\F$ be such that
\begin{equation}\label{xii}
\sum_{i=1}^nx_i=\sum_{j=1}^\ell\beta_j\ind_{B_j},
\end{equation}
with
$$f_0\Bigl(\sum_{i=1}^nx_i\Bigr)=\sum_{j=1}^\ell\beta_j\nu(B_j)\quad\mbox{and}\quad\sum_{j=1}^\ell\abs{\beta_j}\leq (1+\e)\Biggl\|\sum_{i=1}^nx_i\Biggr\|_\F .$$

To estimate the left-hand side of \eqref{Zfm} fix any $x^\ast\in X^\ast$ with $\n{x^\ast}=1$ and observe that the mapping $\F\ni A\mapsto x^\ast\nu(A)_\gamma$ is real-valued and $1$-additive. Hence, by the Kalton--Roberts Theorem \ref{KR}, there exists a~set additive function $\mu_{x^\ast}\colon\F\to\R$ such that $$\abs{x^\ast\nu(A)-\mu_{x^\ast}(A)}\leq K\quad\mbox{for each }A\in\F.$$

Consider the subalgebra $\F^\prime\subset\F$ generated by all $A_{ij}$, for $1\leq i\leq n$, $1\leq j\leq k_i$, and by $B_j$ for $1\leq j\leq\ell$. Since $\F^\prime$ is finite, we may identify it with a power set $2^\Om$, where $\Om$ is the set of all atoms in $\F^\prime$. Factoring each $A_{ij}$ and $B_j$ into elements of $\Om$, and making use of \eqref{xi}, \eqref{xii} and the additivity of $\mu_{x^\ast}$, we obtain
$$\sum_{i=1}^n\sum_{j=1}^{k_i}\alpha_{ij}\mu_{x^\ast}(A_{ij})=\sum_{j=1}^\ell\beta_j\mu_{x^\ast}(B_j).$$ Therefore, the value of the functional $x^\ast$ at the vector under the sign of norm at the left-hand side of \eqref{Zfm} equals $$
D_{x^\ast}:=\sum_{j=1}^\ell\beta_j(x^\ast\nu(B_j)-\mu_{x^\ast}(B_j))-\sum_{i=1}^n\sum_{j=1}^{k_i}\alpha_{ij}(x^\ast\nu(A_{ij})-\mu_{x^\ast}(A_{ij}))$$ Consequently,
\begin{eqnarray*}
\abs{D_{x^\ast}} & \leq & K\Biggl(\sum_{i=1}^n\sum_{j=1}^{k_i}\abs{\alpha_{ij}}+\sum_{j=1}^\ell\abs{\beta_j}\Biggr)\\
& \leq & (1+\e)K\Biggl(\sum_{i=1}^n\n{x_i}_\F+\Biggl\|\sum_{i=1}^nx_i\Biggr\|_\F\Biggr)\leq  2(1+\e)K\sum_{i=1}^n\n{x_i}_\F,
\end{eqnarray*}
which implies \eqref{Zfm}.

Let $\Za_0=X\oplus_{f_0}X_\F$, that is, the direct sum equipped with the quasi-norm $\n{\cdot}_{f_0}$ defined by $$\n{(\xi,x)}_{f_0}=\n{x}_\F+\n{\xi-f_0(x)}.$$ By virtue of the zero-linearity of $f_0$, for every $(\xi_1,x_1),\ldots ,(\xi_n,x_n)\in\Za_0$ we have 
\begin{equation*}
\begin{split}
\Biggl\|\sum_{i=1}^n(\xi_i,x_i)\Biggr\|_{f_0} &=\Biggl\|\sum_{i=1}^nx_i\Biggr\|_\F+\Biggl\|\sum_{i=1}^n\xi_i-f_0\Biggl(\sum_{i=1}^nx_i\Biggr)\Biggr\|\\
& \leq\sum_{i=1}^n\n{x_i}_\F+\Biggl\|\sum_{i=1}^n(\xi_i-f_0(x_i))\Biggr\|+\Biggl\|f_0\Biggl(\sum_{i=1}^nx_i\Biggr)-\sum_{i=1}^nf_0(x_i)\Biggr\|\\
& \leq(1+Z(f_0))\sum_{i=1}^n\n{(\xi_i,x_i)}_{f_0}.
\end{split}
\end{equation*}
Therefore, the formula $$\nn{(\xi,x)}=\inf\Biggl\{\sum_{j=1}^k\n{(\xi_j,x_j)}_{f_0}:\, k\in\N,\,\, (\xi_j,x_j)\in\Za_0\,\mbox{ and }\,(\xi,x)=\sum_{j=1}^k(\xi_j,x_j)\Biggr\}$$ defines a norm on the space $\Za_0$ that satisfies 
\begin{equation}\label{normZ}
(1+Z(f_0))^{-1}\n{(\xi,x)}_{f_0}\leq\nn{(\xi,x)}\leq\n{(\xi,x)}_{f_0}\quad\mbox{for }(\xi,x)\in \Za_0.
\end{equation}

Let $(\Za,\n{\cdot}_\Za)$ be the completion of the space $(\Za_0,\nn{\cdot})$. We shall prove that $\Za$ is a twisted sum of $X$ and $\mathcal{X}_\F$. To this end define $i\colon X\to\Za$ by $i(\xi)=(\xi,0)$ and $q_0\colon\Za_0\to X_\F$ by $q_0(\xi,x)=x$. Then $q_0$ extends uniquely to an operator $q\colon\Za\to\mathcal{X}_\F$ satisfying $\ker(q)=i(X)$. For the last equality first observe that $i(X)\subseteq\ker(q)$, since obviously $\ker(q_0)=i(X)$. For the reverse inclusion suppose $z\in\ker(q)$ and $\lim_{n\to\infty}\n{(\xi_n,x_n)-z}_\Za=0$ for some sequence of $(\xi_n,x_n)\in\Za_0$. Then $$\lim_{n\to\infty}x_n=\lim_{n\to\infty}q(\xi_n,x_n)=q(z)=0,$$ thus $$\lim_{n\to\infty}\n{(f_0(x_n),x_n)}_\Za=\lim_{n\to\infty}\nn{(f_0(x_n),x_n)}\leq\lim_{n\to\infty}\n{(f_0(x_n),x_n)}_{f_0}=0.$$Therefore, we have $(\xi_n-f_0(x_n),0)\longrightarrow z$, and hence the sequence $(\xi_n-f_0(x_n))$, being a~Cauchy sequence, converges to some $\xi_0\in X$ with $z=i(\xi_0)$. Consequently, the sequence 
\begin{equation}\label{exGamma}
\exi{i}{q}{X}{\Za}{\mathcal{X}_\F}
\end{equation}
is exact.

We have already shown that $i$ is an isomorphism onto the closed subspace $i(X)$ of $\Za$; observe also that \eqref{normZ} implies $$\n{i(\xi)}_\Za=\n{(\xi,0)}_\Za=\nn{(\xi,0)}\geq (1+Z(f_0))^{-1}\n{(\xi,0)}_{f_0}=(1+Z(f_0))^{-1}\n{\xi}$$ for each $\xi\in X$, thus $\n{i^{-1}}\leq 1+Z(f_0)$. By Theorem \ref{Pelczynski}, there exists a~norm $\n{\cdot}_\Za^\prime$ on $\Za$ satisfying 
\begin{equation}\label{nprime}
\n{z}_\Za\leq\n{z}_\Za^\prime\leq (1+Z(f_0))\n{z}_\Za
\end{equation}
and such that the identity map $\id\colon (i(X),\n{\cdot}_\Za^\prime)\to X$ is an isometry. Therefore, the space $(\Za,\n{\cdot}_\Za^\prime)$ contains an isometric copy of $X$ and, since the diagram \eqref{exGamma} is an exact sequence, we have also $$(\Za,\n{\cdot}_\Za^\prime)/i(X)\simeq\mathcal{X}_\F.$$

Now, let $\lambda\geq 1$ be such that $X$ is $(\lambda,\kappa)$-injective. Since the density character of $\mathcal{X}_\F$ equals $\max\{\omega,\abs{\F}\}<\kappa$ ($\omega$ in the case where $\F$ is finite), the `moreover' part of Proposition \ref{dens}(d) implies that there exists a~projection $P\colon (\Za,\n{\cdot}_\Za^\prime)\to (i(X),\n{\cdot}_\Za^\prime)$ with $\n{P}\leq 3\lambda$. 

Define a linear map $h\colon X_\F\to X$ by $h(x)=i^{-1}\circ P(0,-x)$. Using \eqref{nprime} and \eqref{normZ} we get 
\begin{eqnarray*}
\begin{split}
\n{f_0(x)-h(x)} &=\n{i^{-1}\circ P(f_0(x),x)}\leq (1+Z(f_0))\n{P(f_0(x),x)}_\Za\\
& \leq (1+Z(f_0))\n{P(f_0(x),x)}_\Za^\prime\leq 3\lambda(1+Z(f_0))\n{(f_0(x),x)}_\Za^\prime\\
& \leq 3\lambda(1+Z(f_0))^2\n{(f_0(x),x)}_\Za=3\lambda(1+Z(f_0))^2\nn{(f_0(x),x)}\\
& \leq 3\lambda(1+Z(f_0))^2\n{(f_0(x),x)}_{f_0}=3\lambda(1+Z(f_0))^2\n{x}_\F\quad\mbox{for }x\in X_\F.
\end{split}
\end{eqnarray*}

Since the concrete form of the zero-linear map $f_0$ was not important in the establishing the existence of $h$, we have actually proved that for every $f\in\Xi(X_\F,X)$ there is a linear map $h\colon X_\F\to X$ such that $\dist(f-h)\leq 3\lambda(1+Z(f))^2$. Therefore, if we denote by $b$ the infimum of all the values of $B$ for which the assertion (iii) of Theorem \ref{stz_Kalton} (applied for quasi-linear maps defined on the dense subspace $X_\F$ of $\mathcal{X}_\F$ and taking values in $X$) holds true, then we infer that $3\lambda(1+x)^2\geq bx$ must hold for every $x\in (0,\infty)$. Hence, $b\leq 12\lambda$, that is, for each $\delta>0$ we may find a linear map $g\colon X_\F\to X$ such that $$\dist(f_0-g)\leq (12\lambda+\delta)Z(f_0)\leq 2(12\lambda+\delta)(1+\e)K.$$

Finally, define $\mu\colon\F\to X$ by $\mu(A)=g(\ind_A)$. Then $\mu$ is a vector measure and $$\n{\nu(A)-\mu(A)}\leq 2(12\lambda+\delta)(1+\e)K\quad\mbox{for each }A\in\F,$$ 
where $\delta$ and $\e$ may be arbitrarily small positive numbers. Hence, we get inequality \ref{kappa_v} and the proof is completed.
\end{proof}

\begin{remark}
The explicit construction of the exact sequence \eqref{exGamma}, instead of a~direct application of the equality $\Ext(\mathcal{X}_\F,X)=0$ (which follows from the $\kappa$-injectivity of $X$ and Proposition \ref{dens}(e)), was needed to make sure that the estimate of $\dist(f_0-g)$ would be independent on the given algebra $\F$. The only thing which should, and which did, play a~role was the cardinality of $\F$. This is also a~reason why we could not use the strategy of extending $f_0$ to a~map $f\in\Xi(\mathcal{X}_\F,X)$ via Theorem \ref{Zextension}; doing this there is no way to keep a~control under $Z(f)$.
\end{remark}
\begin{remark}\label{8cK}
A careful inspection of the above proof shows that the factor $24\lambda$ arises as $2\cdot 4\cdot 3\lambda$, where $2$ comes from the estimate of $Z(f_0)$ via inequality \eqref{Zfm}, $3\lambda$ comes from Proposition \ref{dens}, and $4$ comes from the inequality $3\lambda(1+x)^2\geq bx$. Thus, if we knew that for some constant $c>0$ the underlying Banach space $X$ is $c$-complemented in every Banach space $Z$ such that $Z/X$ has density character less than $\kappa$, then we would get possibly better estimate: $v(\kappa,X)\leq 8cK$. Let us use this observation in what follows.
\end{remark}

The classical Sobczyk theorem \cite{sobczyk} (cf. also \cite[\S 2.5]{albiac_kalton} and the survey \cite{sanchez_castillo_yost}) says that $c_0$ is $2$-complemented in every separable Banach superspace. It was generalised by Hasanov \cite{hasanov} in the way we will now explain.

For a given cardinal number $\Gamma$ and a filter $\G$ of subsets of $\Gamma$ we define a~subspace $c_0(\Gamma,\G)$ of $\ell_\infty(\Gamma)$ by $$c_0(\Gamma,\G)=\bigl\{x\in\ell_\infty(\Gamma)\colon\lim_\G x=0\bigr\}.$$  Call the filter $\G$ a~$\kappa$-{\it filter}, provided that for any collection $\{A_i:\, i\in I\}\subset\G$, with $\abs{I}<\kappa$, we have $\bigcap_{i\in I}A_i\in\G$.
\begin{theorem}[Sobczyk \cite{sobczyk} \& Hasanov \cite{hasanov}]\label{Hasanov}
Let $\Gamma$ and $\kappa$ be cardinal numbers and let $\G$ be a~$\kappa$-filter of subsets of $\Gamma$. Then, the space $c_0(\Gamma,\G)$ is $2$-complemented in every Banach space $Z$ containing it isometrically and such that the density character of $Z/c_0(\Gamma,\G)$ is less than, or equal to $\kappa$.
\end{theorem}

Denote $m_0(\Gamma)$ the space $c_0(\Gamma,\G_\Gamma)$, where $\G_\Gamma=\{A\subset\Gamma\colon\abs{\Gamma\setminus A}<\Gamma\}$ (in particular, $m_0(\omega)=c_0$). Let $\cf(\Gamma)$ stand for the cofinality of $\Gamma$ and $\cf(\Gamma)^+$ for its cardinal successor.

\begin{corollary}\label{cfHasanov}
Let $\Gamma$ be an infinite cardinal number. Then, the space $m_0(\Gamma)$ has the $\cf(\Gamma)^+$-$\SVM$ property, in other words, $\tau(m_0(\Gamma))>\cf(\Gamma)^+$. Moreover, 
\begin{equation}\label{m0_v}
v(\cf(\Gamma)^+,m_0(\Gamma))\leq 16K.
\end{equation}
\end{corollary}
\begin{proof}
If $\{A_i\colon i\in I\}\subset\G_\Gamma$ and $\abs{I}<\cf(\Gamma)\leq\Gamma$ then $\mu:=\sup\{\abs{\Gamma\setminus A_i}\colon i\in I\}<\Gamma$, hence
$$\sum_{i\in I}\abs{\Gamma\setminus A_i}\leq\mu\cdot\abs{I}<\Gamma,$$which shows that $\G_\Gamma$ is a~$\cf(\Gamma)$-filter. Repeating the proof of Theorem \ref{k_injective} for an arbitrary set algebra $\F$ with $\abs{\F}\leq\cf(\Gamma)$ and taking into account Theorem \ref{Hasanov}, jointly with Remark \ref{8cK} (for $c=2$), we get inequality \eqref{m0_v}.
\end{proof}

Recall that for a given compact, Hausdorff space $\Omega$ the derived set $\Omega^\prime$ is defined to be the set of all its accumulation points and, recursively, $\Omega^{(n+1)}=(\Omega^{(n)})^\prime$. We say that $\Omega$ is of {\it finite height}, provided $\Omega^{(n)}=\varnothing$ for some $n\in\N$, in which case we define the ({\it Cantor--Bendixson}) {\it height} of $\Omega$ to be the least such $n$. According to \cite[Proposition 1.24]{aviles}, if $\Omega$ is a~compact, Hausdorff space of height $n$, then the Banach space $\mathcal{C}(\Omega)$ is $(2n-1)$-separably injective. So, we may note the following conclusion:
\begin{corollary}\label{height}
Let $\Omega$ be a compact, Hausdorff space of finite height $n$. Then, the space $\mathcal{C}(\Omega)$ has the $\omega_1$-$\SVM$ property, in other words, $\tau(\mathcal{C}(\Omega))>\omega_1$. Moreover, $$v(\omega_1,\mathcal{C}(\Omega))\leq 24(2n-1)K.$$
\end{corollary}
%%%%%%%%%%%%%%%%%%%%%%%%%%%%%%%%%%%%%%%%%%%
%%%%%%%%%%%%%%%%%%%%%%%%%%%%%%%%%%%%%%%%%%%
\section{The $\SVM$ characters of $m_0(\Gamma)$ and $\JL_\infty$}
\noindent
A natural question which arises in the context of Definition \ref{def_kSVM} is whether for every cardinal number $\kappa\geq\om$ there exists a Banach space $X$ with $\tau(X)=\kappa$. We will partially answer this question positively in the present section by calculating $\tau(m_0(\Gamma))$, which will also allow us to determine $\tau(c_0(\Gamma))$.

According to Corollary \ref{cfHasanov}, we have $\tau(m_0(\Gamma))\geq\cf(\Gamma)^{++}$. To get an upper estimate we shall generalise the Johnson--Lindenstrauss space $\JL_\infty$ (see \cite[Example 2]{johnson_lindenstrauss} and also \cite{yost (JL space)}) in the following way: Instead of using an uncountable almost disjoint family of subsets of $\omega$, let us use a~more general construction, due to Rosenthal.
\begin{proposition}[Rosenthal \cite{rosenthal}]\label{rosenthal}
Let $\Gamma$ be an infinite cardinal number. Then, there exists a~family $\Ra$ of subsets of $\Gamma$ satisfying the following conditions:
\begin{itemize*}
\item[{\rm (i)}] $\abs{\Ra}>\Gamma$;
\item[{\rm (ii)}] $\abs{A}=\Gamma$ for each $A\in\Ra$;
\item[{\rm (iii)}] for every distinct $A,B\in\Ra$ there is an ordinal number $\gamma<\Gamma$ such that for every $\alpha\in A\cap B$ we have $\alpha\leq\gamma$.
\end{itemize*}
\end{proposition}
Decreasing the family $\Ra$, if necessary, we may assume that $\abs{\Ra}=\Gamma^+$. Let $(A_\alpha)_{\alpha\in\Gamma^+}$ be a transfinite sequence of all the elements from $\Ra$. We define a~{\it generalised Johnson--Lindenstrauss space}, $\JL_\infty(\Gamma)$, as the completion (in $\ell_\infty(\Gamma)$) of the space $$\span\bigl(m_0(\Gamma)\cup\bigl\{\ind_{A_\alpha}\colon\alpha\in\Gamma^+\bigr\}\bigr).$$

Clearly, $m_0(\Gamma)$ embeds isometrically into $\JL_\infty(\Gamma)$. Moreover, for every finite linear combination $x=\sum_{j=1}^ka_j\ind_{A_{\alpha_j}}$ we may find a~sequence $x_0\in m_0(\Gamma)$ with $$\n{x_0+x}_{\JL_\infty(\Gamma)}=\max_{1\leq j\leq k}\abs{a_j}.$$This follows from the fact that for each pair of distinct $i,j\in\{1,\ldots ,k\}$ the intersection $A_{\alpha_i}\cap A_{\alpha_j}$ has cardinality less than $\Gamma$ and thus we may choose an element from $m_0(\Gamma)$ with a support contained in this intersection and which annihilates every its coordinate. Repeating this procedure finitely many times we arrive at the required $x_0\in m_0(\Gamma)$. This means that $\JL_\infty(\Gamma)/m_0(\Gamma)\simeq c_0(\Gamma^+)$, and thus we have an exact sequence 
\begin{equation}\label{eJL}
\ex{m_0(\Gamma)}{\JL_\infty(\Gamma)}{c_0(\Gamma^+)}.
\end{equation}

Now, we claim that the sequence above does not split. Observe that each valuation functional $\JL_\infty(\Gamma)\ni x\mapsto x(\gamma)$ (for $\gamma\in\Gamma$) is continuous and hence the family of all such functionals forms a~total set of cardinality $\Gamma$ in $\JL_\infty(\Gamma)^\ast$. This implies that the density character of the \ws\ topology on $\JL_\infty(\Gamma)^\ast$ is at most $\Gamma$. However, $c_0(\Gamma^+)^\ast\simeq\ell_1(\Gamma^+)$ and every subset $B\subset\ell_1(\Gamma^+)$ with $\abs{B}\leq\Gamma$ has an entire support with cardinality not exceeding $\Gamma$, so it cannot be a~total set. This shows that the density character of the \ws\ topology on $c_0(\Gamma^+)^\ast$ is larger than $\Gamma$ and, consequently, $c_0(\Gamma^+)$ cannot be isomorphic to a~subspace of $\JL_\infty(\Gamma)$, which would be the case if our exact sequence split. 

Consequently, the diagram \eqref{eJL} shows that $\Ext(c_0(\Gamma^+),m_0(\Gamma))\not=0$, and hence, by Theorem \ref{T2}(ii), $m_0(\Gamma)$ does not have the $\Gamma^{++}$-$\SVM$ property. In other words, for every infinite cardinal number $\Gamma$ we have $$\tau(m_0(\Gamma))\leq\Gamma^{++},$$which, jointly with Corollary \ref{cfHasanov}, gives the following result:
\begin{theorem}\label{tau(m0)}
For every infinite cardinal $\Gamma$ we have $$\cf(\Gamma)^{++}\leq\tau(m_0(\Gamma))\leq\Gamma^{++}.$$
\end{theorem}

Consequently, if $\Gamma$ is a~regular cardinal (that is, $\cf(\Gamma)=\Gamma$), then we have the equality $\tau(m_0(\Gamma))=\Gamma^{++}$. This gives a partial answer to the question posed at the beginning of this section: Every cardinal number, which is a~double successor of a~regular cardinal, is an $\SVM$ character of some Banach space. 

As a corollary from Theorem \ref{tau(m0)} we obtain $\tau(c_0)=\om_2$, which immediately gives the next result.
\begin{corollary}\label{c0}
For every infinite cardinal $\Gamma$ we have $\tau(c_0(\Gamma))=\om_2$. Moreover, $$v(\omega_1,c_0(\Gamma))\leq 16K.$$
\end{corollary}
\begin{proof}
Since $c_0$ is a complemented subspace of $c_0(\Gamma)$, Proposition \ref{comp} implies that $\tau(c_0(\Gamma))\leq\tau(c_0)=\om_2$. For the reverse inequality suppose $\F$ is a~set algebra with $\abs{\F}\leq\om$ and let $\nu\colon\F\to c_0(\Gamma)$ be a~$1$-additive function. Pick any $\e>0$ and for each $A\in\F$ choose a~finite set $\Gamma(\e,A)\subset\Gamma$ such that $\abs{\nu(A)(\gamma)}<\e$ for $\gamma\in\Gamma\setminus\Gamma(\e,A)$. Let $\Gamma_0=\bigcup_{A\in\F}\Gamma(\e,A)$; it is a~countable set, so the subspace $\{\xi\in c_0(\Gamma):\, \supp(\xi)\subset\Gamma_0\}$ of $c_0(\Gamma)$ is isometrically isomorphic to $c_0$ (or to a~finite-dimensional space $\ell_\infty(\Gamma_0)$, if $\Gamma_0$ is finite). Therefore, for every $c>16K$ Theorem \ref{cfHasanov} produces a~vector measure $\mu\colon\F\to c_0(\Gamma)$ such that $$\supp(\mu(A))\subset\Gamma_0\,\,\,\mbox{ and }\,\,\,\n{\nu(A)(\gamma)-\mu(A)(\gamma)}\leq c\quad\mbox{for every }A\in\F,\,\gamma\in\Gamma_0.$$ Hence, if we let $\e<16K$, then $\n{\nu(A)-\mu(A)}\leq c$ for each $A\in\F$, which proves that $c_0(\Gamma)$ has the $\omega_1$-$\SVM$ property with $v(\om_1,c_0(\Gamma))\leq 16K$.
\end{proof}

Now, we will focus on the Johnson--Lindenstrauss space $$\JL_\infty=\overline{\span}^{\n{\cdot}_\infty}\bigl(c_0\cup\bigl\{\ind_{N_\gamma}\colon\gamma\in\Gamma\bigr\}\bigr),$$
where $\{N_\gamma\}_{\gamma\in\Gamma}$ is an uncountable almost disjoint family of infinite subsets of $\omega$, which will be fixed for the rest of this section. We wish to derive the following result:
\begin{theorem}\label{JLchar}
We have $\tau(\JL_\infty)=\omega_2$ and 
\begin{equation}\label{120K}
v(\omega_1,\JL_\infty)\leq 120K.
\end{equation}
\end{theorem}
\begin{proof}
It is well-known that $\JL_\infty$ is separably injective, since it is a~(non-trivial) twisted sum of two separably injective Banach spaces: $c_0$ and $c_0(\Gamma)$ (cf. \cite{johnson_lindenstrauss} and \cite{yost (JL space)}), whereas `being separably injective' is a~$3$SP property (cf. \cite{aviles}). Hence, Theorem \ref{k_injective} implies that $\tau(\JL_\infty)\geq\omega_2$. To get an~upper estimate for $v(\omega_1,\JL_\infty)$ we shall take a~look at the spectral representation of the commutative unital Banach algebra $$\JL_\infty^{(1)}=\JL_\infty\oplus_1\R\ind_\omega\quad\mbox{(the }\ell_1\mbox{-sum)}.$$We may safely replace $\JL_\infty$ by $\JL_\infty^{(1)}$ in our assertion, since the former is $1$-complemented in the latter.

As it is explained in \cite{yost (JL space)}, $\JL_\infty^{(1)}$ is isometrically isomorphic to $\mathcal{C}(\Omega)$, where $$\Omega=\omega\cup\Gamma\cup\{\infty\}$$ is the one-point compactification of the set $\omega\cup\Gamma$ (the disjoint union) topologised as follows:
\begin{itemize*}
\item each point of $\omega$ is a~discrete point;
\item each point $\gamma\in\Gamma$ has a~neighbourhood basis consisting of all the sets $U\subset\omega\cup\Gamma$ for which $\abs{N_\gamma\setminus U}<\omega$.
\end{itemize*}
Evidently, $\Omega^{\prime\prime}=\{\infty\}$, thus the height of $\Omega$ equals $3$ and so inequality \eqref{120K} follows from Corollary \ref{height}.

To complete the proof we shall only show that $\tau(\JL_\infty^{(1)})\leq\omega_2$. To this end we may simply observe that $\Omega$ is scattered, which guarantees that the space $\mathcal{C}(\Omega)$ contains an~isomorphic copy of $c_0$ that is complemented in $\mathcal{C}(\Omega)$ (see \cite[Theorem 14.26]{fabian}). Thus, it remains to appeal to Proposition \ref{comp} and Corollary \ref{c0}.
\end{proof}

There is, however, another way of showing that $\tau(\JL_\infty)\leq\omega_2$, without appealing to the existence of a~complemented copy of $c_0$ inside $\JL_\infty$, but making use of Theorem \ref{T2}(ii). This job will be done by constructing a~non-trivial twisted sum of $\JL_\infty$ and $c_0(\omega_1)$. The idea is taken from an~old construction, due to Foia\c{s} and Singer \cite{foias_singer}, of a~non-trivial twisted sum $\mathcal{D}$ of $\mathcal{C}[0,1]$ and $c_0$. The space $\mathcal{D}$ consists of all real functions on $[0,1]$ that are continuous at every point outside some prescribed countable dense set $Q\subset [0,1]$, and left continuous with right limits at every point from $Q$. The supremum norm makes $\mathcal{D}$ a~Banach space containing $\mathcal{C}[0,1]$ isometrically, so that $\mathcal{D}/\mathcal{C}[0,1]\simeq c_0$. However, the quotient map $\pi\colon\mathcal{D}\to\mathcal{D}/\mathcal{C}[0,1]$ not only does not admit any lifting in the class of operators $c_0\to\mathcal{D}$, but even does not admit any uniformly continuous, not necessarily linear, lifting (see \cite[Example 1.20]{benyamini_lindenstrauss}).
\begin{proposition}\label{ext_JL}
$\Ext(c_0(\omega_1),\JL_\infty)\not=0$.
\end{proposition}
\begin{proof}
We may again work with $\JL_\infty^{(1)}$ instead of $\JL_\infty$, because for any Banach spaces $X$, $Y_1$ and $Y_2$ we have: $\Ext(X,Y_1\oplus Y_2)\not=0$ if and only if $\Ext(X,Y_i)\not=0$ for some $i\in\{1,2\}$ (see \cite[Lemma 4]{sanchez_castillo (duality)}). 

We may also assume that $\abs{\Gamma}=\omega_1$, because even if the given almost disjoint family $\{N_\gamma\}_{\gamma\in\Gamma}$ has cardinality (consistently) greater than $\omega_1$, we may carry out the construction below for any its subfamily with cardinality $\omega_1$.

For every $\gamma\in\Gamma$ write $N_\gamma=P_\gamma\cup Q_\gamma$, where $P_\gamma$ and $Q_\gamma$ are disjoint and infinite. The set $P_\gamma$ will play a~role of the `left side' of $\gamma$, while $Q_\gamma$ will be the `right side' of $\gamma$. Let $\Omega=\omega\cup\Gamma\cup\{\infty\}$ be as described above. Define $\mathcal{Z}$ to be the space of all the functions $f\colon\Omega\to\R$ such that:
\begin{itemize*}
\item $f$ is bounded;
\item $f$ is continuous at the point $\infty$;
\item at each point $\gamma\in\Gamma$ the `left limit' $\lim_{n\in P_\gamma}f(n)$ exists and equals $f(\gamma)$ ({\it i.e.} for every $\e>0$ the set $\{n\in P_\gamma\colon\abs{f(n)-f(\gamma)}\geq\e\}$ is finite);
\item at each point $\gamma\in\Gamma$ the `right limit' $\lim_{n\in Q_\gamma}f(n)$ exists, but may be possibly different from $f(\gamma)$ ({\it i.e.} there is $\lambda\in\R$ such that for every $\e>0$ the set $\{n\in Q_\gamma\colon\abs{f(n)-\lambda}\geq\e\}$ is finite).
\end{itemize*}
Let us write $f(\gamma-)$ instead of $\lim_{n\in P_\gamma}f(n)$ ($=f(\gamma)$) and $f(\gamma+)$ instead of $\lim_{n\in Q_\gamma}f(n)$, for any $f\in\mathcal{Z}$ and any $\gamma\in\Gamma$.

The supremum norm makes $\mathcal{Z}$ a~Banach space that contains $\mathcal{C}(\Omega)\simeq\JL_\infty^{(1)}$ isometrically. Now, consider the `jump function' $J$ defined by $$\mathcal{Z}\ni f\xmapsto[\phantom{xxxxxx}]{J}\left(\frac{1}{2}\bigl(f(\gamma+)-f(\gamma-)\bigr)\colon\gamma\in\Gamma\right).$$
{\it Claim 1. }$Jf\in c_0(\Gamma)$ for every $f\in\mathcal{Z}$.

\vspace*{1mm}\noindent
{\it Proof of Claim 1. }Every open neighbourhood of $\infty$ is of the form $V\cup\{\infty\}$, where $V$ is open in $\omega\cup\Gamma$ and has a~compact complement. It is also easily seen that a~set $S\subset\omega\cup\Gamma$ is compact if and only if:
\begin{itemize*}
\item $S\cap\Gamma$ is finite, say $\{\gamma_1,\ldots,\gamma_k\}$; 
\item $S$ contains only finitely many points outside the set $\omega\setminus\bigcup_{j=1}^kN_{\gamma_j}$.
\end{itemize*}
In particular, every open neighborhood of $\{\infty\}$ contains all but finitely many $\gamma\in\Gamma$, and for all but finitely many $\gamma$'s it must contain almost all elements from $N_\gamma$.

Now, fix any $\e>0$. If we had $\abs{e_\gamma^\ast Jf}>\e$ for infinitely many $\gamma$'s from $\Gamma$, then the remarks above would imply that the oscillation of $f$ is at least $2\e$ on every open neighbourhood of $\{\infty\}$. This is, however, impossible as $f$ is continuous at $\infty$. {\it Claim 1} has been proved.

\vspace*{2mm}\noindent
{\it Claim 2. }The map $J\colon\mathcal{Z}\to c_0(\Gamma)$ is surjective and 
\begin{equation}\label{distJL}
\n{Jf}_\infty=\dist(f,\mathcal{C}(\Omega))\quad\mbox{for every }f\in\mathcal{Z}.
\end{equation}
Consequently, the quotient $\mathcal{Z}/\mathcal{C}(\Omega)$ is isometrically isomorphic to $c_0(\Gamma)$ (so, also to $c_0(\omega_1)$), hence we have the exact sequence: 
\begin{equation}\label{exiJL}
\exi{}{J}{\mathcal{C}(\Omega)}{\mathcal{Z}}{c_0(\Gamma)}.
\end{equation}

\vspace*{1mm}\noindent
{\it Proof of Claim 2. }Let us start with showing \eqref{distJL}. Of course, we may suppose $f\not\in\mathcal{C}(\Omega)$. The inequality `$\leq$' is evident. For the reverse inequality pick any positive $\e<\dist(f,\mathcal{C}(\Omega))$ and an~open neighbourhood $U$ of $\infty$ on which the oscillation of $f$ is less than $\e$. Then, the complement $\Omega\setminus U$ contains only finitely many elements from $\Gamma$, say $\gamma_1,\ldots ,\gamma_k$. By decreasing $U$, if necessary, we may assume that $\bigcup_{j=1}^kN_{\gamma_j}\subset\Omega\setminus U$, which makes $U$ a~clopen subset of $\Omega$. The set $\Omega\setminus U$ may still contain some natural numbers outside $\bigcup_{j=1}^kN_{\gamma_j}$, but only finitely many of them. Denote $M$ the set of all these numbers. Let also $F=\bigcup_{1\leq i<j\leq k}(N_{\gamma_i}\cap N_{\gamma_j})$ which is a~finite set. Then we may write $\Omega\setminus U$ as the disjoint union: $$\Omega\setminus U=\{\gamma_1,\ldots ,\gamma_k\}\cup\bigcup_{j=1}^k(N_{\gamma_j}\setminus F)\cup F\cup M.$$For each $1\leq j\leq k$ pick a~finite set $S_j\subset N_{\gamma_j}\setminus F$ such that
\begin{equation}\label{eJL}
\bigl|f(n)-f(\gamma_j-)\bigr|<\e\,\,\,\mbox{ or }\,\,\,\bigl|f(n)-f(\gamma_j+)\bigr|<\e\quad\mbox{for each }n\in N_{\gamma_j}\setminus (F\cup S_j).
\end{equation}

Now, define a function $g\colon\Omega\to\R$ as follows:
\begin{itemize*}
\item $g(\gamma_j)=f(\gamma_j)+e_{\gamma_j}^\ast Jf$ for each $1\leq j\leq k$;
\item $g(n)=g(\gamma_j)$ for each $n\in N_{\gamma_j}\setminus (F\cup S_j)$;
\item $g(n)=f(n)$ for each $n\in\bigcup_{j=1}^kS_j\cup F\cup M$;
\item $g(x)=f(\infty)$ for each $x\in U$.
\end{itemize*}
Since $U$ is clopen, $g$ is a continuous function, $g\in\mathcal{C}(\Omega)$. Note also that:
\begin{itemize*}
\item $\abs{g(\gamma_j)-f(\gamma_j)}=\abs{e_{\gamma_j}^\ast Jf}\leq\n{Jf}_\infty$ for each $1\leq j\leq k$;
\item $g(n)=g(\gamma_j)=f(\gamma_j)+e_{\gamma_j}^\ast Jf=\frac{1}{2}\bigl(f(\gamma_j-)+f(\gamma_j+)\bigr)$ for each $n\in N_{\gamma_j}\setminus(F\cup S_{\gamma_j})$, so \eqref{eJL} implies that for all such $n$'s we have $\abs{g(n)-f(n)}<\e+\abs{e_{\gamma_j}^\ast Jf}\leq\e+\n{Jf}_\infty$;
\item $\abs{g(n)-f(n)}=0$ for each $n\in\bigcup_{j=1}^kS_j\cup F\cup M$;
\item $\abs{g(x)-f(x)}=\abs{f(\infty)-f(x)}<\e$ for every $x\in U$.
\end{itemize*}
This shows that $\n{g-f}_\infty<\e+\n{Jf}_\infty$ and, as $\e$ could be arbitrarily small, the inequality `$\geq$' of \eqref{distJL} follows. In particular, $J$ is a~bounded operator.

Let $\Phi\colon\mathcal{Z}/\mathcal{C}(\Omega)\to c_0(\Gamma)$ be the linear map satisfying $\Phi\pi=J$, where $\pi\colon\mathcal{Z}\to\mathcal{Z}/\mathcal{C}(\Omega)$ is the canonical projection. Plainly, each vector $e_\gamma$ from the canonical Schauder basis $(e_\gamma)_{\gamma\in\Gamma}$ of $c_0(\Gamma)$ belongs to the range of $J$, so it belongs also to the range of $\Phi$, which is therefore dense in $c_0(\Gamma)$. Equality \eqref{distJL} says that $\Phi$ is an~isometry, thus $\Phi$ is in fact surjective, and hence so is $J$. {\it Claim 2} has been proved.

\vspace*{1mm}
Now, in order to show that the exact sequence \eqref{exiJL} does not split, observe that the Banach space $\mathcal{Z}$ embeds isometrically into a~direct sum of two Johnson--Lindenstrauss type spaces determined by the two almost disjoint families: $\{P_\gamma\}_{\gamma\in\Gamma}$ and $\{Q_\gamma\}_{\gamma\in\Gamma}$. More precisely, let
$$
\JL_{\infty,P}=\overline{\span}^{\n{\cdot}_\infty}\bigl(c_0\cup\bigl\{\ind_{P_\gamma}\colon\gamma\in\Gamma\bigr\}\bigr)
\quad\mbox{and}\quad \JL_{\infty,Q}=\overline{\span}^{\n{\cdot}_\infty}\bigl(c_0\cup\bigl\{\ind_{Q_\gamma}\colon\gamma\in\Gamma\bigr\}\bigr)
$$
and
$$
\JL_{\infty,P}^{(1)}=\JL_{\infty,P}\oplus_1\R\ind_\omega\quad\mbox{and}\quad\JL_{\infty,Q}^{(1)}=\JL_{\infty,Q}\oplus_1\R\ind_\omega.
$$
Let also $\Omega_P$ and $\Omega_Q$ be the spectra of these two commutative unital Banach algebras (they are equal to $\Omega$ as sets); their topologies may be described as the one on $\Omega$, just replacing $\{N_\gamma\}_{\gamma\in\Gamma}$ by $\{P_\gamma\}_{\gamma\in\Gamma}$ and $\{Q_\gamma\}_{\gamma\in\Gamma}$, respectively. Each function $f\in\mathcal{Z}$ induces in a~natural way two functions $f_P\in\mathcal{C}(\Omega_P)$ and $f_Q\in\mathcal{C}(\Omega_Q)$, defined as follows:
\begin{itemize*}
\item $f_P(\infty)=f(\infty)=f_Q(\infty)$;
\item $f_P(\gamma)=f(\gamma-)=f(\gamma)$ and $f_Q(\gamma)=f(\gamma+)$ for each $\gamma\in\Gamma$;
\item $f_P(n)=f(n)$ for $n\in\bigcup_{\gamma\in\Gamma}P_\gamma$ and $f_Q(n)=f(n)$ for $n\in\bigcup_{\gamma\in\Gamma}Q_\gamma$;
\item $f_P(n)=0$ for $n\in\omega\setminus\bigcup_{\gamma\in\Gamma}P_\gamma$ and $f_Q(n)=0$ for $n\in\omega\setminus\bigcup_{\gamma\in\Gamma}Q_\gamma$.
\end{itemize*}
Therefore, we have an isometric embedding $j$ and the inclusion embedding $i$ as below:
$$
\mathcal{Z}\ni f\xmapsto[\phantom{xxxxxx}]{j}(f_P,f_Q)\in\mathcal{C}(\Omega_P)\oplus_{\scriptscriptstyle{\infty}}\mathcal{C}(\Omega_Q)\simeq\JL_{\infty,P}^{(1)}\oplus_{\scriptscriptstyle{\infty}}\JL_{\infty,Q}^{(1)}\xhookrightarrow[\phantom{xxxxxx}]{i}\ell_\infty\oplus_{\scriptscriptstyle{\infty}}\ell_\infty\simeq\ell_\infty.
$$

To finish the argument note that the space $\mathcal{Z}$, being isomorphic to a~subspace of $\ell_\infty$, has a~$w^\ast$-separable dual, whereas $c_0(\omega_1)^\ast\simeq\ell_1(\omega_1)$ is not $w^\ast$-separable, as every countable set of functionals has a~countable entire support on $\omega_1$. Consequently, $c_0(\omega_1)$ is not isomorphic to any subspace of $\mathcal{Z}$, thus the exact sequence \eqref{exiJL} does not split and the proof is completed.
\end{proof}
%%%%%%%%%%%%%%%%%%%%%%%%%%%%%%%%%%%%%%%%%%%
%%%%%%%%%%%%%%%%%%%%%%%%%%%%%%%%%%%%%%%%%%%
\section{The three-space problem for the $\kappa$-$\SVM$ property}
\noindent
Let $X$ and $Y$ be normed spaces. Following \cite{sanchez_castillo (stability)} we denote $K(X,Y)$ the `approximation constant` for quasi-linear maps from $\Lambda(X,Y)$, {\it i.e.} the infimum of those constants $B\leq\infty$ such that for every $f\in\Lambda(X,Y)$ there exists a~linear map $h\colon X\to Y$ with $\dist(f,h)\leq B\cdot\Delta(f)$. Similarly, let $Z(X,Y)$ be the infimum of those constants $B\leq\infty$ such that for every $f\in\Xi(X,Y)$ there exists a~linear map $h\colon X\to Y$ with $\dist(f,h)\leq B\cdot Z(f)$.

By Theorems \ref{st_Kalton} and \ref{stz_Kalton}, if $X$ and $Y$ are Banach spaces, then $(X,Y)$ splits if and only if $K(X_0,Y)<\infty$ for some (and then, for every) dense subspace $X_0$ of $X$, whereas $\Ext(X,Y)=0$ if and only if $Z(X_0,Y)<\infty$ for some (and then, for every) dense subspace $X_0$ of $X$.

We are aiming for the following theorem:
\begin{theorem}\label{3sp}
If $\kappa$ is a cardinal number with uncountable cofinality, then the $\kappa$-$\SVM$ property is a~$\mathsf{3SP}$ property. Consequently, the $\SVM$ property is a~$\mathsf{3SP}$ property.
\end{theorem}

Our strategy will be based on the homological result saying that for every Banach space $\mathfrak{X}$ the property having the form $\Ext(\mathfrak{X},\cdot)=0$ is a~$\mathsf{3SP}$ property. It follows from the fact that every short exact sequence 
\begin{equation}\label{jqyz}
\exi{j}{q}{Y}{Z}{X}
\end{equation}
of Banach spaces $X$, $Y$, $Z$ induces a~long homology sequence
\begin{equation*}
\begin{split}
0 &\longrightarrow\mathscr{B}(\mathfrak{X},Y)\overset{j_\ast}{\longrightarrow}\mathscr{B}(\mathfrak{X},Z)\overset{q_\ast}{\longrightarrow}\mathscr{B}(\mathfrak{X},X)\overset{\theta}{\longrightarrow}\\
&\overset{\theta}{\longrightarrow}\Ext(\mathfrak{X},Y)\overset{\alpha}{\longrightarrow}\Ext(\mathfrak{X},Z)\overset{\beta}{\longrightarrow}\Ext(\mathfrak{X},X)\longrightarrow 0
\end{split}
\end{equation*} 
in the category of vector spaces and linear maps ($\mathscr{B}(E,F)$ is the set of all bounded, linear operators from $E$ to $F$). Here $j_\ast$ is composing with $j$ on the left; $q_\ast$ is composing with $q$ on the left; $\theta(T)$ is the lower row from the pull-back applied to \eqref{jqyz} and any given operator $T\in\mathscr{B}(\mathfrak{X},X)$; $\alpha$ produces a~push-out from any given element from $\Ext(\mathfrak{X},Y)$ and the operator $j$; $\beta$ produces a~push-out from any given element from $\Ext(\mathfrak{X},Z)$ and the operator $q$. For a~detailed description of long homology sequences, see \cite{sanchez_castillo (homology)}.

To put our plan into action let us start with a~quantitative counterpart of Theorem \ref{T2} for the spaces $\mathcal{X}_\F$ introduced in Section 5.
\begin{theorem}\label{quant}
Let $\kappa$ be a cardinal number and $X$ be a Banach space. If $X$ has the $\kappa$-$\SVM$ property then for every set algebra $\F$ with $\abs{\F}<\kappa$ the pair $(\mathcal{X}_\F,X)$ splits and, moreover,
\begin{equation}\label{Kquant}
K(X_\F,X)\leq 4(2+v(\kappa,X)).
\end{equation}
\end{theorem}
Before proving this theorem let us note two simple facts.
\begin{lemma}\label{step}
Let $\F$ be any set algebra and let $t\in X_\F$, $t=\sum_{j=1}^Na_j\ind_{A_j}$, where $a_j\in\R$ and $A_j\in\F$ {\rm (}for $1\leq j\leq N${\rm )} are pairwise disjoint. Then, for every set $S\subset\{1,\ldots ,N\}$ we have
\begin{equation}\label{SF}
\Biggl\|\sum_{j\in S}a_j\ind_{A_j}\Biggr\|_\F\leq\Biggl\|\sum_{j=1}^Na_j\ind_{A_j}\Biggr\|_\F.
\end{equation}
Consequently, putting $t^+=\max\{t,0\}$ and $t^-=-\min\{t,0\}$ we have $t^+,t^-\in X_\F$, $\n{t^+}_\F\leq\n{t}_\F$ and $\n{t^-}_\F\leq\n{t}_\F$.
\end{lemma}
\begin{proof}
For any $n\in\N$, any scalars $b_j$, and any sets $B_j\in\F$ ($1\leq j\leq n$) such that $t=\sum_{j=1}^nb_j\ind_{B_j}$ we have $$\sum_{j\in S}a_j\ind_{A_j}=\sum_{j=1}^nb_j\ind_{B_j\cap S}.$$Hence, passing to the infimum over all such representations of $t$, we obtain inequality \eqref{SF}.

Next, observe that
$$
t^+=\sum_{\{j\colon a_j>0\}}a_j\ind_{A_j}\in X_\F\quad\mbox{ and }\quad t^-=\sum_{\{j\colon a_j<0\}}a_j\ind_{A_j}\in X_\F
$$
thus the two desired inequalities follow now from \eqref{SF}.
\end{proof}
\begin{lemma}\label{ru}
Let $\F$ be any set algebra, $X$ be a~Banach space and $r\colon X_\F\to X$ be a~quasi-linear map with $\delta=\Delta(r)$. Let also $u=\sum_{j=1}^Nb_j\ind_{A_j}$, where $b_j\in\R$ and $A_j\in\F$ {\rm (}for $1\leq j\leq N${\rm )} are pairwise disjoint. Then 
$$
\Biggl\|r(u)-\sum_{j=1}^Nb_jr\bigl(\ind_{A_j}\bigr)\Biggr\|\leq 2(N-1)\delta\n{u}_\F.
$$
\end{lemma}
\begin{proof}
Denote $$d_k=\Biggl\|r(u)-\sum_{j=1}^kb_jr(\ind_{A_j})\Biggr\|\quad\mbox{for }1\leq k\leq N.$$
Using the quasi-linearity of $r$ recursively $(N-1)$ times, and applying Lemma~\ref{step}, we get
\begin{equation*}
\begin{split}
d_N &\leq\Biggl\|r(u)-r\Biggl(\sum_{j=1}^{N-1}b_j\ind_{A_j}\Biggr)-r\bigl(b_N\ind_{A_N}\bigr)\Biggr\|+\Biggl\|r\Biggl(\sum_{j=1}^{N-1}b_j\ind_{A_j}\Biggr)-\sum_{j=1}^{N-1}b_jr\bigl(\ind_{A_j}\bigr)\Biggr\|\\
&\leq\delta\Biggl(\Biggl\|\sum_{j=1}^{N-1}b_j\ind_{A_j}\Biggr\|_\F+\bigl\|b_N\ind_{A_N}\bigr\|_\F\Biggr)+d_{N-1}\\
&\overset{\eqref{SF}}{\leq}2\delta\n{u}_\F+d_{N-1}\leq\ldots\leq 2(N-1)\delta\n{u}_\F.
\end{split}
\end{equation*}
\end{proof}

\begin{proof}[Proof of Theorem \ref{quant}]
Let $\F\subset 2^\Omega$ be any set algebra with $\abs{\F}<\kappa$ and let $v>v(\kappa,X)$. We are to prove that for every quasi-linear map $f\colon X_\F\to X$, with $\delta=\Delta(f)$, there exists a~linear map $h\colon X_\F\to X$ such that 
\begin{equation*}
\n{f(t)-h(t)}\leq 4\delta(2+v)\n{t}_\F\quad\mbox{for }t\in X_\F.
\end{equation*}
Proceeding likewise in the proof of Theorem \ref{T1} we get a~linear map $h$ which satisfies $\n{r(\ind_A)}\leq 2\delta v$ for $A\in\F$, where $r=f-h$ is quasi-linear with $\Delta(r)\leq\delta$.

Lemma \ref{step} implies that for every $t\in X_\F$ we have $$\n{r(t)}\leq\delta\bigl(\n{t^+}_\F+\n{t^-}_\F\bigr)+\n{r(t^+)}+\n{r(t^-)}.$$Hence, by homogeneity, it is enough to show that for every $t\in X_\F$ with $0\leq t\leq\ind_\Omega$ we have $\n{r(t)}\leq\delta(3+2v)$. So, fix any such $t$. We may write $t=\sum_{j=1}^Na_j\ind_{A_j}$, where $0\leq a_j\leq 1$ and $A_j\in\F$ are pairwise disjoint ($1\leq j\leq N$). 

For every $k\in\N$ define $$B_k=\bigcup\bigl\{A_j\colon 2^{-k}\mbox{ appears in the binary representation of }a_j\bigr\}$$(for rational numbers we choose representations with infinitely many terms). Let $$s_m=\sum_{k=1}^m\frac{1}{2^k}\ind_{B_k}\quad\mbox{for }m\in\N.$$From the very definition of the norm $\n{\cdot}_\F$ we get
\begin{equation}\label{sm1}
\n{s_m}_\F\leq\sum_{k=1}^m\frac{1}{2^k}<1\quad\mbox{for }m\in\N.
\end{equation}
Obviously, the distance between $t$ and $s_m$ in the supremum norm is at most $2^{-m}$, but observe also that $t-s_m$ is constant on each of the sets $A_j$ ($1\leq j\leq N$), say $t-s_m=\sum_{j=1}^Nb_j\ind_{A_j}$. This shows that
\begin{equation}\label{t-sm}
\n{t-s_m}_\F\leq N\cdot 2^{-m}\quad\mbox{for every }m\in\N.
\end{equation}
We may also apply Lemma \ref{ru} to $u=t-s_m$ getting
\begin{equation*}
\Biggl\|r(t-s_m)-\sum_{j=1}^Nb_jr\bigl(\ind_{A_j}\bigr)\Biggr\|\leq 2(N-1)\delta\n{t-s_m}_\F\leq 2^{-m+1}(N-1)N\delta.
\end{equation*}
Hence, 
\begin{equation}\label{r(t-sm)}
\begin{split}
\n{r(t-s_m)} &\leq 2^{-m+1}(N-1)N\delta+\Biggl\|\sum_{j=1}^Nb_jr\bigl(\ind_{A_j}\bigr)\Biggr\|\\
&\leq 2^{-m+1}(N-1)N\delta+\sum_{j=1}^N\abs{b_j}\cdot\n{r\bigl(\ind_{A_j}\bigr)}\\
&\leq 2^{-m+1}(N-1)N\delta+2^{-m+1}N\delta v=:\e_m
\end{split}
\end{equation}

Likewise in the proof of Theorem \ref{T1}, we have $$\Biggl\|r(s_m)-\sum_{k=1}^m{1\over 2^k}r(\ind_{B_k})\Biggr\|\leq\delta\sum_{k=1}^m{k\over 2^k}\leq 2\delta ,$$thus 
\begin{equation}\label{sm}
\n{r(s_m)}\leq 2\delta(1+v)\quad\mbox{for }m\in\N.
\end{equation}
Consequently, using \eqref{r(t-sm)}, \eqref{sm}, \eqref{t-sm} and \eqref{sm1}, we obtain
\begin{equation*}
\begin{split}
\n{r(t)} &\leq\n{r(t-s_m)}+\n{r(s_m)}+\delta\bigl(\n{t-s_m}_\F+\n{s_m}_\F\bigr)\\
&\leq\e_m+2\delta(1+v)+\delta(N\cdot 2^{-m}+1)\xrightarrow[m\to\infty]{\phantom{xxxxxx}}\delta(3+2v),
\end{split}
\end{equation*}
as required.
\end{proof}

\begin{remark}\label{localK}
What we have really proved is the following `local' version of Theorem \ref{quant}: If $\F$ is a~set algebra, $X$ is a~Banach space, and $v(\F,X)$ is the infimum of all those constants $v\leq\infty$ such that for every $1$-additive function $\nu\colon\F\to X$ there exists a~vector measure $\mu\colon\F\to X$ with $\n{\nu(A)-\mu(A)}\leq v$ for each $A\in\F$ (we allow $v=\infty$), then $K(X_\F,X)\leq 4(2+v(\F,X))$. Hence, $(\mathcal{X}_\F,X)$ splits provided that $v(\F,X)<\infty$.
\end{remark}

\begin{corollary}\label{cor_iff}
Let $\kappa$ be a cardinal number and $X$ be a~Banach space. Then, the following assertions are equivalent:
\begin{itemize*}
\item[{\rm (i)}] $X$ has the $\kappa$-$\SVM$ property;
\item[{\rm (ii)}] $K(X_\F,X)\leq B$ for some $B<\infty$ and every set algebra $\F$ with $\abs{\F}<\kappa$;
\item[{\rm (iii)}] $Z(X_\F,X)\leq C$ for some $C<\infty$ and every set algebra $\F$ with $\abs{\F}<\kappa$.
\end{itemize*}
\end{corollary}
\begin{proof}
The implication (i)$\Rightarrow$(ii) follows from Theorem \ref{quant}. The implication (ii)$\Rightarrow$(iii) is evident. The implication (iii)$\Rightarrow$(i) is hidden in the proof of Theorem \ref{k_injective}. In fact, we have shown therein that for any set algebra $\F$ with $\abs{\F}<\kappa$ and any $\e>0$ every $1$-additive function $\nu\colon\F\to X$ induces a~zero-linear map $f_0\in\Xi(X_\F,X)$ with $Z(f_0)\leq 2(1+\e)K$, which in turn induces a~locally convex twisted sum of $X$ and $\mathcal{X}_\F$. By condition (iii), for every $\delta>0$ there exist a~linear map $h\colon X_F\to X$ with $$\dist(f_0,h)\leq (C+\delta)Z(f_0)\leq 2(C+\delta)(1+\e)K,$$thus the vector measure $\mu\colon\F\to X$ defined by $\mu(A)=g(\ind_A)$ satisfies $$\n{\nu(A)-\mu(A)}\leq 2(C+\delta)(1+\e)K\quad\mbox{for each }A\in\F,$$which shows that $X$ has the $\kappa$-$\SVM$ property with $v(\kappa,X)\leq 2CK$.
\end{proof}

\begin{remark}\label{Zlocal}
The same reasoning as above leads to the following `local' version of the implication (iii)$\Rightarrow$(i) from Corollary \ref{cor_iff}: For every set algebra $\F$, every~Banach space $X$ and every $1$-additive function $\nu\colon\F\to X$ there exists a~vector measure $\mu\colon\F\to X$ such that $$\n{\nu(A)-\mu(A)}\leq 2K\cdot Z(X_\F,X)\quad\mbox{for each }A\in\F.$$
\end{remark}

\begin{proof}[Proof of Theorem \ref{3sp}]
Given an exact sequence \eqref{jqyz} suppose, towards a~contradiction, that $X$ and $Y$ have the $\kappa$-$\SVM$ property, while $Z$ does not. Then, by Theorem \ref{quant}, we have
\begin{equation*}
\Ext(\mathcal{X}_\F,X)=0\,\,\mbox{ and }\,\,\Ext(\mathcal{X}_\F,Y)=0\quad\mbox{for every set algebra }\F\mbox{ with }\abs{\F}<\kappa.
\end{equation*}
Hence, $\Ext(\mathcal{X}_\F,Z)=0$ for every set algebra $\F$ with $\abs{\F}<\kappa$. Therefore, by Remark \ref{Zlocal}, for any such set algebra $\F$ there exists a~constant $v<\infty$ such that for every $1$-additive function $\nu\colon\F\to Z$ there is a~vector measure $\mu\colon\F\to Z$ satisfying $\n{\nu(A)-\mu(A)}\leq v$ for $A\in\F$. Let $v(\F,Z)$ be the infimum of all such constants $v$. By our supposition, there exist a~sequence $(\F_n)_{n=1}^\infty$ of set algebras with $\kappa_n:=\abs{\F_n}<\kappa$ and such that for every $n\in\N$ we have $v(\F_n,Z)>n$; let $\nu_n\colon\F_n\to Z$ be a~$1$-additive function witnessing the last inequality. Since $\cf(\kappa)>\omega$, we have $\lambda:=\sup_n\kappa_n<\kappa$.

Now, regard each $\F_n$ as a~Boolean ring $(\F_n,\striangle,\cap)$ with identity $\Omega_n$ (that is, a~Boolean algebra) and let $\Sigma\subset\sprod_{n=1}^\infty\F_n$ be a~subring of the simple product of Boolean algebras, generated by the set $$\bigcup\bigl\{j_m(\F_m)\colon m\in\N\bigr\}\cup\bigl\{(\Omega_1,\Omega_2,\ldots)\bigr\},$$where $j_m\colon\F_m\to\sprod_{n=1}^\infty\F_n$ is the canonical injection. Plainly, $\Sigma$ is a~Boolean subalgebra of $\sprod_{n=1}^\infty\F_n$ and $\abs{\Sigma}=\max\{\omega,\lambda\}<\kappa$. However, the maps $\nu_m\circ\pi_m|_\Sigma\colon\Sigma\to Z$ (where $\pi_m\colon\sprod_{n=1}^\infty\F_n\to\F_m$ are the canonical projections), for $m\in\N$, witness that $v(\Sigma,Z)=\infty$; a~contradiction.
\end{proof}
%%%%%%%%%%%%%%%%%%%%%%%%%%%%%%%%%%%%%%%%%%%
%%%%%%%%%%%%%%%%%%%%%%%%%%%%%%%%%%%%%%%%%%%
\section{Characterisation of the $\SVM$ property for Banach spaces complemented in their biduals}
\noindent
Now, we give a characterization of the $\SVM$ property in the class of Banach spaces which are complemented in their bidual. Recall first the following result, \cite[Theorem 1.11]{jebreen_jamjoom_yost}:
\begin{theorem}[Jebreen, Jamjoom \&\ Yost \cite{jebreen_jamjoom_yost}]\label{JJY_theorem}
For any Banach spaces $Y$ and $Z$ the vector spaces $\Ext(Z,Y^\ast)$ and $\Ext(Y,Z^\ast)$ are isomorphic.
\end{theorem}
\begin{theorem}\label{T4}
Let $X$ be a Banach space complemented in its bidual. Then the following assertions are equivalent:
\begin{itemize*}
\item[{\rm (i)}] $X$ has the $\SVM$ property;
\item[{\rm (ii)}] $\mathrm{Ext}(X^\ast,\ell_1)=0$;
\item[{\rm (iii)}] $\mathrm{Ext}(\ell_\infty, X^{\ast\ast})=0$;
\item[{\rm (iv)}] $\mathrm{Ext}(c_0,X)=0$.
\end{itemize*}
\end{theorem}
\begin{proof}
The implication (iv)$\Rightarrow$(ii) follows from Theorem \ref{pulgarin_thm}. By Theorem \ref{JJY_theorem}, we have $\Ext(\ell_\infty,X^{\ast\ast})=\Ext(X^\ast,\ell_1^{\ast\ast})$, thus (iii)$\Rightarrow$(ii) as $\ell_1$ is complemented in its bidual. Similarly, we have $\Ext(X^\ast,\ell_1)=\Ext(c_0,X^{\ast\ast})$, hence (ii)$\Rightarrow$(iv) as $X$ is complemented in its bidual. Moreover, by Theorem \ref{T2}(ii), we infer that (i)$\Rightarrow$(iv) and also that (i) implies $\Ext(\ell_\infty,X)=0$ which in turn easily implies (iii) as $X$ is complemented in $X^{\ast\ast}$. Hence, the only thing left to be proved is the implication (ii)$\Rightarrow$(i).

Since $X$ is assumed to be complemented in $X^{\ast\ast}$, Proposition \ref{comp} says that $X$ has the $\SVM$ property whenever $X^{\ast\ast}$ does. Moreover, in view of Proposition \ref{bidual}, it is enough to show that $X^{\ast\ast}$ has the $\omega$-$\SVM$ property. So, let us fix any finite set algebra $\F\subset 2^\Om$ and any $1$-additive function $\nu\colon\F\to X^{\ast\ast}$. We may assume that $\F=2^\Omega$. The subspace of $X^{\ast\ast\ast}$ consisting of all $w^\ast$-continuous functionals on $X^{\ast\ast}$ is identified with $X^\ast$, so for any $x^\ast\in X^\ast$ let us write $x^\ast\circ\nu$ for the map $\F\ni A\mapsto\nu(A)x^\ast$.

Let $\M$ be the subspace of $\ell_\infty(\F)$ consisting of all set additive functions. Define an operator $T\colon\M\to\ell_1(\Om)$ by $T(m)=(m\{\om\})_{\om\in\Om}$. Clearly, we have$${1\over 2}\n{T(m)}_1\leq\n{m}_\infty\leq\n{T(m)}_1\quad\mbox{for }m\in\M .$$By virtue of the Kalton--Roberts Theorem \ref{KR}, for each $x^\ast\in X^\ast$, $\n{x^\ast}=1$, there is a~measure $\mu_{x^\ast}\in\M$ such that
\begin{equation}\label{T41}
\bigl|(x^\ast\circ\nu)(A)-\mu_{x^\ast}(A)\bigr|\leq K\n{x^\ast}\quad\mbox{for }A\in\F .
\end{equation}
For each pair $\{x^\ast,-x^\ast\}$ replace, if necessary, $\mu_{-x^\ast}$ by $-\mu_{x^\ast}$, and for every $\lambda\in\R$ define $\mu_{\lambda x^\ast}=\lambda\mu_{x^\ast}$. In this manner, we obtain a family $\{\mu_{x^\ast}:\, x^\ast\in X^\ast\}$ satisfying \eqref{T41} for every $x^\ast\in X^\ast$ and such that the mapping $\hat\mu\colon X^\ast\to\M$, defined by $\hat\mu(x^\ast)=\mu_{x^\ast}$, is homogeneous. Moreover, for all $x_1^\ast,\ldots ,x_n^\ast\in X^\ast$ and $A\in\F$ we have 
\begin{equation*}
\begin{split}
\Biggl|\hat\mu\Biggl(\sum_{k=1}^n x_k^\ast\Biggr)(A)&-\sum_{k=1}^n\hat\mu(x_k^\ast)(A)\Biggr|\\ 
&\leq\bigl|\mu_{x_1^\ast+\ldots+x_n^\ast}(A)-(x_1^\ast+\ldots+x_n^\ast)\circ\nu(A)\bigr|+\sum_{k=1}^n\bigl|\mu_{x_k^\ast}(A)-x_k^\ast\circ\nu(A)\bigr|\\
&\leq K\n{x_1^\ast+\ldots+x_n^\ast}+K\sum_{k=1}^n\n{x_k^\ast}\leq 2K\sum_{k=1}^n\n{x_k^\ast},
\end{split}
\end{equation*}
which shows that $\hat\mu$ is a zero-linear map with $Z(\hat\mu)\leq 2K$. Consequently, the composition $T\circ\hat\mu\colon X^\ast\to\ell_1(\Om)$ is also zero-linear and $Z(T\circ\hat\mu)\leq 4K$. Since $\ell_1(\Om)$ is a $1$-complemented subspace of $\ell_1$, there is a linear map $h\colon X^\ast\to\ell_1(\Om)$ such that $$\n{(T\circ\hat\mu)(x^\ast)-h(x^\ast)}\leq 4Z(X^\ast,\ell_1)K\n{x^\ast}\quad\mbox{for }x^\ast\in X^\ast.$$Defining a~linear map $\Psi\colon X^\ast\to\M$ by $\Psi=T^{-1}\circ h$, we obtain 
\begin{equation}\label{T42}
\n{\hat\mu(x^\ast)-\Psi(x^\ast)}\leq 4Z(X^\ast,\ell_1)K\n{x^\ast}\quad\mbox{for }x^\ast\in X^\ast .
\end{equation}

For each $\alpha\in\Om$ let $\Psi_\alpha\colon X^\ast\to\R$ be given as $\Psi_\alpha(x^\ast)=\Psi(x^\ast)\{\alpha\}$. Denote $$r_\alpha=\bigl(K+4Z(X^\ast,\ell_1)K+\n{\nu\{\alpha\}}\bigr)\n{x^\ast}.$$Then $\n{\Psi_\alpha}\leq r_\alpha$, and thus $\Psi_\om\in X^{\ast\ast}$, which follows from the following estimate:
\begin{eqnarray*}
\abs{\Psi_\alpha(x^\ast)}&\leq & \abs{\hat\mu(x^\ast)\{\alpha\}}+\n{\hat\mu(x^\ast)-\Psi(x^\ast)}_\infty\\
&\leq & \abs{(x^\ast\circ\nu)\{\alpha\}-\mu_{x^\ast}\{\alpha\}}+\abs{(x^\ast\circ\nu)\{\alpha\}}+4Z(X^\ast,\ell_1)K\n{x^\ast}\\
&\leq & K\n{x^\ast}+\n{\nu\{\alpha\}}\cdot\n{x^\ast}+4Z(X^\ast,\ell_1)K\n{x^\ast}=r_\alpha\n{x^\ast}\,\,\mbox{ for }x^\ast\in X^\ast.
\end{eqnarray*}

Now, define $\mu\colon\F\to X^{\ast\ast}$ by $$\mu(A)=\sum_{\alpha\in A}\Psi_\alpha\quad\mbox{for }A\in\F.$$
Then, for every $A\in\F$ and $x^\ast\in X^\ast$ we have
\begin{equation*}
\begin{split}
\mu(A)x^\ast &=\sum_{\alpha\in A}\Psi_\alpha x^\ast=\sum_{\alpha\in A}(T^{-1}\circ h)(x^\ast)\{\alpha\}=(T^{-1}\circ h)(x^\ast)(A),
\end{split}
\end{equation*}
because $(T^{-1}\circ h)(x^\ast)\in\M$. Consequently, $$(\mu(A\cup B)-\mu(A)-\mu(B)  )x^\ast=0\quad\mbox{for }x^\ast\in X^\ast\mbox{ and }A,B\in\F ,\,\, A\cap B=\varnothing,$$which means that $\mu$ is a~vector measure.

Finally, for every $x^\ast\in X^\ast$ and $A\in\F$ inequalities \eqref{T41} and \eqref{T42} yield
\begin{equation*}
\begin{split}
\abs{(\nu(A)&-\mu(A))x^\ast}=\abs{(x^\ast\circ\nu)(A)-(T^{-1}\circ h)(x^\ast)(A)}\\
&\leq\abs{(x^\ast\circ\nu)(A)-\hat\mu(x^\ast)(A)}+\n{\hat\mu(x^\ast)-\Psi(x^\ast)}\leq (1+4Z(X^\ast,\ell_1))K\n{x^\ast}.
\end{split}
\end{equation*}
Since the set of all $w^\ast$-continuous functionals on $X^{\ast\ast}$ is norming, our result follows with 
$$
v(X)\leq\n{\pi}v(X^{\ast\ast})\quad\mbox{ and }\quad v(X^{\ast\ast})\leq (1+4Z(X^\ast,\ell_1))K,
$$
where $\pi\colon X^{\ast\ast}\to X$ is any bounded projection onto $X$ (recall Propositions \ref{bidual} and \ref{comp}, and the fact that $X^{\ast\ast}$ is $1$-complemented in its bidual, as it is a~dual space).
\end{proof}

\begin{corollary}\label{svm_bidual}
Let $X$ be a~Banach space. Then $X^{\ast\ast}$ has the $\SVM$ property if and only if $\Ext(X^\ast,\ell_1)=0$.
\end{corollary}
\begin{proof}
In view of Theorem \ref{JJY_theorem}, $\Ext(X^\ast,\ell_1)=0$ is equivalent to $\Ext(c_0,X^{\ast\ast})=0$ which, by Theorem \ref{T4}, is in turn equivalent to $X^{\ast\ast}$ having the $\SVM$ property
\end{proof}

In view of Theorem \ref{T4}, it is interesting to ask what is the class of all Banach spaces $Y$ satisfying
\begin{equation}\label{Yl1}
\Ext(Y,\ell_1)=0.
\end{equation}
By the Lindenstrauss Lifting Principle, this condition holds true for any $\mathscr{L}_1$-space $Y$. More generally, for every $\mathscr{L}_1$-space $Y$ and every Banach space $Z$, complemented in $Z^{\ast\ast}$, we have $\Ext(Y,Z)=0$ (see \cite[Proposition 2.1]{kalton_pelczynski}). Up to the best of my knowledge, there is no example known of a~non-$\mathscr{L}_1$-space $Y$ satisfying \eqref{Yl1}. If such an example does not exist, then:
\begin{itemize*}
\item Every Banach space $X$ having the $\omega_1$-$\SVM$ property would be a~$\mathscr{L}_\infty$-space, which is equivalent both to $X^\ast$ being a~$\mathscr{L}_1$-space and to $X^{\ast\ast}$ being injective (see \cite[Chapter~5]{lindenstrauss_tzafriri}). Indeed, by Theorem \ref{T2}(ii), if $X$ has the $\omega_1$-$\SVM$ property then $\Ext(c_0,X)=0$, hence $\Ext(X^\ast,\ell_1)=0$.
\item Every bidual Banach space $X^{\ast\ast}$ having the $\SVM$ property would be necessarily injective (this follows from Corollary \ref{svm_bidual}).
\end{itemize*}

We shall present a partial result concerning condition \eqref{Yl1}. Before this, let us recall some definitions. Every exact sequence of the form
\begin{equation}\label{qP}
\exi{j}{q}{\mathrm{ker}(q)}{P}{Y},
\end{equation}
with $P$ being a~projective Banach space (that is, $P\simeq\ell_1(\Gamma)$ for some index set $\Gamma$), is called a~{\it projective presentation} of $Y$. Following \cite{kalton (locally)} we say that a~subspace $E$ of a~Banach space $X$ is {\it locally complemented} if there exists a~constant $\lambda$ such that for every finite-dimensional space $F\subset X$ and any $\e>0$ there is an~operator $T_F\colon F\to E$ such that $\n{T_F}\leq\lambda$ and $\n{T_Fx-x}<\e$ for each $x\in E\cap F$. We say that an~exact sequence $\exi{i}{\pi}{E}{X}{Y}$ {\it locally splits} whenever $i(E)$ is locally complemented in $X$. By a~standard compactness argument, and the Banach--Alaoglu theorem, this is equivalent to saying that the dual sequence $\exi{\pi^\ast}{i^\ast}{Y^\ast}{X^\ast}{E^\ast}$ splits (more directly, that the quotient operator $i^\ast$ admits a~right inverse).

Let us also recall the following result which is \cite[Proposition 3.1]{kalton_pelczynski} and, in its `uniform' version, \cite[Proposition 1]{sanchez_castillo (uniform)}.
\begin{theorem}[Kalton \&\ Pe\l czy\'nski \cite{kalton_pelczynski}, Cabello S\'anchez \&\ Castillo \cite{sanchez_castillo (uniform)}]\label{KalPel}
Given Banach spaces $Y$ and $Z$, and a~projective presentation \eqref{qP} of $Y$, the following conditions are equivalent:
\begin{itemize*}
\item[{\rm (i)}] $\Ext(Y,Z)=0$;
\item[{\rm (ii)}] there exists a~constant $c$ such that every operator $t\colon\mathrm{ker}(q)\to Z$ extends to an operator $T\colon P\to Z$ such that $\n{T}\leq c\n{t}$.
\end{itemize*} 
\end{theorem}
\begin{theorem}\label{proj}
Assume $Y$ is a~Banach space satisfying \eqref{Yl1} and having a~projective presentation \eqref{qP} with $\mathrm{ker}(q)$ being a~$\mathscr{L}_1$-space. Then $Y$ itself is a~$\mathscr{L}_1$-space.
\end{theorem}
\begin{proof}
First, observe that by condition \eqref{Yl1}, Theorem \ref{KalPel}, and the fact that for each $m\in\N$ the space $\ell_1^m$ is $1$-complemented in $\ell_1$,
\begin{equation}\label{l1m}
\begin{array}{l}
\mbox{there exists a constant }c\geq 1\mbox{ such that for each }m\in\N\mbox{ every operator}\\
t\colon\mathrm{ker}(q)\to\ell_1^m\mbox{ extends to an operator }T\colon P\to\ell_1^m\mbox{ with }\n{T}\leq c\n{t}.
\end{array}
\end{equation}
Since $\mathrm{ker}(q)$ is a $\mathscr{L}_1$-space, 
\begin{equation}\label{kerqL}
\begin{array}{l}
\mbox{there exists a constant }\rho>1\mbox{ such that for every finite-dimensional subspace }G\\
\mbox{of }\mathrm{ker}(q)\mbox{ there is a finite-dimensional subspace }H\mbox{ of }\mathrm{ker}(q)\mbox{ such that }G\subset H,\\
d_\mathsf{BM}(H,\ell_1^m)<\rho\mbox{ (where }m=\dim H\mbox{), and there is a projection }\pi\colon\mathrm{ker}(q)\to H\\
\mbox{with }\n{\pi}\leq\rho
\end{array}
\end{equation}
(see \cite[Proposition II.5.9]{lindenstrauss_tzafriri}). 

In order to show that \eqref{qP} locally splits, fix any finite-dimensional subspace $F\subset P$ and put $G=F\cap\mathrm{ker}(q)$. Let $H$ be as in \eqref{kerqL} and let $S\colon H\to\ell_1^m$ be an isomorphism satisfying $\n{S}\cdot\n{S^{-1}}\leq\rho$. Now, apply \eqref{l1m} to the operator $t=S\pi\colon\mathrm{ker}(q)\to\ell_1^m$, $\n{t}\leq\n{S}\rho$. We get an extension $T\colon P\to\ell_1^m$ of $t$ with $\n{T}\leq\n{S}c\rho$. Define $T_F=S^{-1}T\colon P\to H$. Then, $\n{T_F}\leq c\rho^2$ and for each $x\in G$ we have $T_Fx=S^{-1}Tx=S^{-1}S\pi x=\pi x=x$, which shows that \eqref{qP} locally splits. Consequently, the dual sequence $$\exi{q^\ast}{j^\ast}{Y^\ast}{P^\ast}{\mathrm{ker}(q)^\ast}$$splits, which means that $Y^\ast$ is complemented in $P^\ast\simeq\ell_\infty(\Gamma)$, so it is injective, hence $Y$ itself is a~$\mathscr{L}_1$-space.
\end{proof}
\begin{remark}\label{proj2}
The assumption about $\mathrm{ker}(q)$ is fully legitimate, because if $Y$ actually is a~$\mathscr{L}_1$-space, then the kernel of every projective presentation of $Y$ must be a~$\mathscr{L}_1$-space as well (consult \cite[Proposition II.5.13]{lindenstrauss_tzafriri} for the separable case; the non-separable case is basically the same).
\end{remark}
%%%%%%%%%%%%%%%%%%%%%%%%%%%%%%%%%%%%%%%%%%%
%%%%%%%%%%%%%%%%%%%%%%%%%%%%%%%%%%%%%%%%%%%
\section{Final remarks}
\noindent
Let us list some problems which arise from this paper.
\setenumerate{itemsep=1em,leftmargin=2em}
\begin{enumerate}
\item Let $\kappa\geq\omega_1$ be a~cardinal number and let $\lambda\geq 1$. In view of Theorem \ref{k_injective}, one may ask whether the class of all Banach spaces $X$ having the $\kappa$-$\SVM$ property is richer than the class of $(\lambda,\kappa)$-injective spaces. A~glance at Proposition \ref{dens}(e) and Corollary \ref{cor_iff} shows that these two classes are equal if and only if the following two assertions are equivalent:
\begin{itemize*}
\item $Z(Y,X)\leq B$ for some $B<\infty$ and every Banach space $Y$ with density character less than $\kappa$;
\item$Z(X_\F,X)\leq C$ for some $C<\infty$ and every set algebra $\F$ with $\abs{\F}<\kappa$.
\end{itemize*}
If $\cf(\kappa)>\omega$, these two assertions are equivalent to $\Ext(Y,X)=0$ and $\Ext(X_\F,X)=0$, respectively, with the same restrictions upon $Y$ and $\F$.

\item Does the $\SVM$ property implies injectivity? For a~possible counterexample one may ask whether the spaces $\ell_\infty/c_0$ and/or $\ell_\infty^c(\Gamma)$ (the space of all countably supported sequences from $\ell_\infty(\Gamma)$, with $\Gamma$ being uncountable) have the $\SVM$ property. These two are both known to be universally separably injective, yet not injective (see \cite{aviles}).

\item Of particular interest is the case $\kappa=\omega_1$ in (1). If it was true that every Banach space $X$ having the $\omega_1$-$\SVM$ property is $\omega_1$-injective (that is, separably injective), then every such space would be automatically a~$\mathscr{L}_\infty$-space and, if only $X$ is infinite-dimensional, it would contain a~copy of $c_0$ (see \cite[Proposition 1.8]{aviles}). We may therefore ask: Must every Banach space having the $\omega_1$-$\SVM$ property be a~$\mathscr{L}_\infty$-space? By our remarks concerning condition \eqref{Yl1} and Theorem \ref{proj}, we know that the answer is positive for Banach spaces $X$ such that the kernel of a~projective presentation of $X^\ast$ is a~$\mathscr{L}_1$-space. Similarly, we may ask the next question:

\item Assume $X$ is an infinite-dimensional Banach space having the $\omega_1$-$\SVM$ property (or even the $\SVM$ property). Does it imply that $X$ contains a~copy of $c_0$? Notice that, in view of Theorem \ref{T4}, a~positive answer (even under the stronger assumption) would imply that $\Ext(c_0,Y)\not=0$ for every infinite-dimensional, reflexive Banach space $Y$, {\it i.e.} every such space would produce a~non-trivial locally convex twisted sum with $c_0$. Cabello S\'anchez and Castillo showed that $\Ext(c_0,Y)\not=0$ whenever $Y$ is a~cotype $2$ space, complemented in its bidual (\cite[Corollary 1]{sanchez_castillo (uniform)}).

\item A weaker variant of (2) and (3): Does the $\SVM$ property of $X$ imply that $X^{\ast\ast}$ is injective?

\item The assumption $\cf(\kappa)>\omega$ was made in Theorem \ref{3sp}, because although we know that the property $\Ext(\mathfrak{X},\cdot)=0$ is a~$\mathsf{3SP}$ property, we do not know if it is a~$\mathsf{3SP}$ property `uniformly'. More precisely, the question reads as follows: Is there a~function $\varphi\colon (0,\infty)^2\to (0,\infty)$ such that whenever $\mathfrak{X}$, $X$, $Y$, $Z$ are Banach spaces, $\ex{Y}{Z}{X}$ is an exact sequence, and $\Ext(\mathfrak{X},X)=0=\Ext(\mathfrak{X},Y)$,
then $$\,\,\,\Ext(\mathfrak{X},Z)=0\quad\mbox{and}\quad Z(\mathfrak{X},Z)\leq\varphi\bigl(Z(\mathfrak{X},X),Z(\mathfrak{X},Y)\bigr)\mbox{?}$$
If the answer is `yes', then we may drop the assumption $\cf(\kappa)>\omega$ in Theorem \ref{3sp}, by using Corollary \ref{cor_iff}. Note that the homological proof of the fact that $\Ext(\mathfrak{X},\cdot)=0$ is a~$\mathsf{3SP}$ property does not touch this issue as it identifies all zero-linear maps which are approximable by linear ones, no matter what their internal structure is.
\end{enumerate}
%%%%%%%%%%%%%%%%%%%%%%%%%%%%%%%%%%%%%%%%%%%
%%%%%%%%%%%%%%%%%%%%%%%%%%%%%%%%%%%%%%%%%%%
\section*{Acknowledgements}{I am grateful to the anonymous Referee for several valuable remarks.}
%%%%%%%%%%%%%%%%%%%%%%%%%%%%%%%%%%%%%%%%%%


\begin{thebibliography}{99}
\bibitem{albiac_kalton} F. Albiac, N.J. Kalton, \textit{Topics in Banach Space Theory}, Graduate Texts in Mathematics 233, Springer 2006.
%\bibitem{aoki} T. Aoki, \textit{Locally bounded linear topological spaces}, Proc. Imp. Acad. Tokyo~18 (1942), no. 10.
%\bibitem{argyros_godefroy_rosenthal} S.A. Argyros, G.~Godefroy, H.P.~Rosenthal \textit{Descriptive set theory and Banach spaces}. In: Handbook of the Geometry of Banach Spaces, vol.~2, pages 1007--1069. North-Holland, Amsterdam 2003.
\bibitem{argyros_todorcevic} S.A. Argyros, S. Todorcevic, \textit{Ramsey Methods in Analysis}, Advanced Courses in Mathematics. CRM Barcelona, Birkh\"auser Verlag, Basel--Boston--Berlin 2005.
\bibitem{aviles} A. Avil\'es, F. Cabello S\'anchez, J.M.F.~Castillo, M.~Gonz\'alez, Y.~Moreno, \textit{On separably injective Banach spaces}, Adv. Math.~234 (2013), 192--216.
\bibitem{benyamini_lindenstrauss} Y. Benyamini, J. Lindenstrauss, \textit{Geometric Nonlinear Functional Analysis}, vol.~1, Amer. Math. Soc., Colloquium Publications~48, Providence, Rhode Island~2000.
\bibitem{bessaga_pelczynski} C. Bessaga, A. Pe\l czy\'nski, \textit{Spaces of continuous functions IV (On isomorphic classifications of spaces $C(S)$)}, Studia Math.~19 (1960), 53--62.
\bibitem{sanchez} F. Cabello S\'anchez, \textit{Yet another proof of Sobczyk's theorem}. In: Methods in Banach Space Theory, London Math. Soc. Lecture Notes~337, pages~133--138. Cambridge Univ. Press~2006.
\bibitem{sanchez_castillo (duality)} F. Cabello S\'anchez, J.M.F. Castillo, \textit{Duality and twisted sums of Banach spaces}, J. Funct. Anal.~175 (2000), 1--16.
\bibitem{sanchez_castillo (dissertationes)} F. Cabello S\'anchez, J.M.F. Castillo, \textit{Banach space techniques underpinning a theory for nearly additive mappings}, Dissertationes Math.~404 (2002), 73 pp.
\bibitem{sanchez_castillo (uniform)} F. Cabello S\'anchez, J.M.F. Castillo, \textit{Uniform boundedness and twisted sums of Banach spaces}, Houston J. Math.~30 (2004), 523--536.
\bibitem{sanchez_castillo (homology)} F. Cabello S\'anchez, J.M.F. Castillo, \textit{The long homology sequence for quasi-Banach spaces, with applications}, Positivity~8 (2004), 379--394.
\bibitem{sanchez_castillo (stability)} F. Cabello S\'anchez, J.M.F. Castillo, \textit{Stability constants and the homology of quasi-Banach spaces}, preprint.
\bibitem{sanchez_castillo_kalton_yost} F. Cabello S\'anchez, J.M.F. Castillo, N.J. Kalton, D.T. Yost, \textit{Twisted sums with~$C(K)$ spaces}, Trans. Amer. Math. Soc.~355 (2003), 4523--4541.
\bibitem{sanchez_castillo_yost} F. Cabello S\'anchez, J.M.F. Castillo, D. Yost, \textit{Sobczyk's theorems from A to B}, Extracta Math.~15 (2000), 391--420.
\bibitem{casazza_lin_lohman} P.G. Casazza, B.-L. Lin, R.H. Lohman, \textit{On James' quasi-reflexive Banach space}, Proc. Amer. Math. Soc.~67 (1977), 265--271.
\bibitem{casazza_shura} P.G. Casazza, T.J. Shura, \textit{Tsirelson's space}, Lecture Notes in Mathematics~1363, Springer-Verlag 1989. 
\bibitem{castillo} J.M.F. Castillo, \textit{Snarked sums of Banach spaces}, Extracta Math.~12 (1997), 117--128.
\bibitem{castillo_gonzalez} J.M.F. Castillo, M. Gonz\'alez, \textit{Three-space Problems in Banach Space Theory}, Lecture Notes in Mathematics~1667, Springer 1997.
\bibitem{ciesielski} K. Ciesielski, \textit{Set theory for the working mathematician}, Cambridge University Press~1997.
\bibitem{dew} N. Dew, \textit{Asymptotic Structure of Banach Spaces}, Ph.D dissertation, University of Oxford, St.~John's College~2002.
\bibitem{diestel_uhl} J. Diestel, J.J. Uhl, \textit{Vector measures}, American Mathematical Society, Providence, R.I. 1977. With a foreword by B.J. Pettis, Mathematical Surveys, No.~15.
\bibitem{enflo_lindenstrauss_pisier} P. Enflo, J. Lindenstrauss, G.~Pisier, \textit{On the \lq\lq three-space problem\rq\rq}, Math. Scand.~36 (1975), 199--210.
\bibitem{fabian} M. Fabian, P. Habala, P.~H\'ajek, V.~Montesinos, V.~Zizler, \textit{Banach Space Theory. The Basis for Linear and Nonlinear Analysis}, CMS Books in Mathematics, Springer~2011.
\bibitem{forest} H. Fetter, B. Gamboa de Buen, \textit{The James Forest}, London Math. Soc. Lecture Note Series~236, Cambridge Univ. Press, Cambridge~1997.
\bibitem{figiel_johnson} T. Figiel, W.B. Johnson, \textit{A~uniformly convex Banach space which contains no $l_p$}, Compositio Math.~29 (1974), 179--190.
\bibitem{foias_singer} C. Foia\c{s}, I. Singer, \textit{On bases in $C([0,1])$ and $L^1([0,1])$}, Rev. Roum. Math. Pures Appl.~10 (1965), 931--960.
\bibitem{forti} G.L. Forti, \textit{Hyers--Ulam stability of functional equations in several variables}, Aeq. Math.~50 (1995), 143--190.
%\bibitem{gajda} Z. Gajda, \textit{On stability of additive mappings}, Internat. J. Math. Math. Sci.~14 (1991), 431--434.
%\bibitem{ger} R. Ger, \textit{The singular case in the stability behaviour of linear mappings}. In: General Inequalities~6, W. Walter (ed.), Internat. Ser. Numer. Math.~103, pages 227--240, Birkh\"auser, Basel 1992.
%\bibitem{grafakos} L. Grafakos, \textit{Classical Fourier Analysis} (2nd edition), Graduate Texts in Mathematics~249, Springer 2008.
\bibitem{hasanov} V.S. Hasanov, \textit{Some universally complemented subspaces in $m(\Gamma)$}, Math. Zametki~27 (1980), 105--108.
\bibitem{hyers} D.H. Hyers, \textit{On the stability of the linear functional equation}, Proc. Nat. Acad. Sci. U.S.A.~271 (1941), 222--224.
%\bibitem{jamjoom_jebreen} F.B.H. Jamjoom, H.M. Jebreen, \textit{On twisted sums of Banach spaces}, Far East J.~Math. Sci.~26 (2007), 169--184.
\bibitem{jebreen_jamjoom_yost} H.M. Jebreen, F.B.H. Jamjoom, D. Yost, \textit{Colocality and twisted sums of Banach spaces}, J. Math. Anal. Appl.~323 (2006), 864--875.
\bibitem{johnson_lindenstrauss} W.B. Johnson, J. Lindenstrauss, \textit{Some remarks on weakly compactly generated Banach spaces}, Israel J. Math.~17 (1974), 219--230.
%\bibitem{kaijser} S. Kaijser, \textit{A note on dual Banach spaces}, Math. Scand.~41 (1977), 325--330.
\bibitem{kalton} N.J. Kalton, \textit{The three space problem for locally bounded $\displaystyle{F}$-spaces}, Compositio Math.~37 (1978), 243--276.
\bibitem{kalton (problem)} N.J. Kalton, Problem~5 (p. 284), Measure Theory and Its Applications (Proc. Conf., Northern Illinois Univ. 1980, G.A. Goldin and R.F. Wheeler (ed.)), DeKalb, Illinois 1981.
\bibitem{kalton (locally)} N.J. Kalton, \textit{Locally complemented subspaces and $\mathscr{L}_p$-spaces for $0<p<1$}, Math. Nachr.~115 (1984), 71--97.
%\bibitem{kalton (section)} N.J. Kalton, \textit{Quasi-Banach spaces}. In: Handbook of the Geometry of Banach Spaces, vol.~2, pages 1099--1130. North-Holland, Amsterdam 2003.
\bibitem{kalton_peck} N.J. Kalton, N.T. Peck, \textit{Twisted sums of sequence spaces and the three space problem}, Trans. Amer. Math. Soc.~255 (1979), 1--30.
%\bibitem{kalton_peck (quotients)} N.J. Kalton, N.T. Peck, \textit{Quotients of $L_p(0,1)$ for $0\leq p<1$}, Studia Math.~64 (1979), 65--75.
\bibitem{kalton_pelczynski} N.J. Kalton, A.~Pe\l czy\'nski, \textit{Kernels of surjections from $\mathscr{L}_1$-spaces with an application to Sidon sets}, Math. Ann.~309 (1997), 135--158.
\bibitem{kalton_roberts} N.J. Kalton, J.W. Roberts, \textit{Uniformly exhaustive submeasures and nearly additive set functions}, Trans. Amer. Math. Soc.~278 (1983), 803--816.
%\bibitem{kelley} J.L. Kelley, \textit{Measures on Boolean algebras}, Pacific J. Math.~9 (1959), 1165--1177.
%\bibitem{lacey} H.E. Lacey, \textit{Isometric theory of classical Banach spaces}, Grundlehren Math. Wiss., Band 208, Springer-Verlag, Berlin 1973.
%\bibitem{lima_yost} \AA . Lima, D. Yost, \textit{Absolutely Chebyshev subspaces}, Proc. Centre Math. Anal. Austral. Nat. Univ.~20 (1988), 116--127.
\bibitem{lindenstrauss_pelczynski} J. Lindenstrauss, A. Pe\l czy\'nski, \textit{Absolutely summing operators in $\mathcal{L}_p$-spaces and their applications}, Studia Math.~29 (1968), 275--326.
\bibitem{lindenstrauss_rosenthal} J. Lindenstrauss, H.P. Rosenthal, \textit{The $\mathcal{L}_p$ spaces}, Israel J. Math.~7 (1969), 325--349.
\bibitem{lindenstrauss_tzafriri} J. Lindenstrauss, L. Tzafriri, \textit{Classical Banach Spaces}, Lecture Notes in Mathematics~338, Springer-Verlag 1973. 
\bibitem{miljutin} A.A. Miljutin, \textit{Isomorphism of the spaces of continuous functions over compact sets of the cardinality of the continuum}, Teor. Funkci\u{i} Funkcional. Anal. Prilo\v{z}en. Vyp.~2 (1966), 150--156. (1 foldout). (Russian)
\bibitem{pawlik} B. Pawlik, \textit{Approximately additive set functions}, Colloq. Math.~54 (1987), 163--164.
\bibitem{pelczynski} A. Pe\l czy\'nski, \textit{Projections in certain Banach spaces}, Studia Math.~19 (1960), 209--228.
%\bibitem{phillips} R.S. Phillips, \textit{On linear transformations}, Trans. Amer. Math. Soc.~48 (1940), 516--541.
%\bibitem{pippenger} N. Pippenger, \textit{Superconcentrators}, SIAM J. Comput.~6 (1977), 298--304.
\bibitem{polya_szego} G. P\'olya, G. Szeg\H{o}, \textit{Problems and Theorems in Analysis}~I, Springer 1976.
%\bibitem{pulgarin} A.A. Pulgar\'{i}n, \textit{Characterization of dual extensions in the category of Banach spaces}, Divulgaciones Mat.~7 (1999), 133--142.
%\bibitem{ribe} M. Ribe, \textit{Examples for the nonlocally convex three space problem}, Proc. Amer. Math. Soc.~73 (1979), 351--355.
%\bibitem{roberts} J.W. Roberts, \textit{A nonlocally convex $F$-space with the Hahn--Banach approximation property}. In: Banach Spaces of Analytic Functions, pages 76--81, Lecture Notes in Mathematics~604, Springer, Berlin 1977.
%\bibitem{rolewicz} S. Rolewicz, \textit{On certain classes of linear metric spaces}, Bull. Acad. Polon. Sci.~5 (1957), 471--473.
\bibitem{rosenthal} H.P. Rosenthal, \textit{On relatively disjoint families of measures, with some applications to Banach space theory}, Studia Math.~37 (1970), 13--36.
\bibitem{rosenthal (chapter)} H.P. Rosenthal, \textit{The Banach spaces $C(K)$}. In: Handbook of the Geometry of Banach Spaces, vol.~2, pages 1547--1602. North-Holland, Amsterdam 2003.
%\bibitem{samuel} C. Samuel, \textit{Indice de Szlenk des $C(K)$ ($K$ espace topologique compact d\'e nombrable}, Seminar on the Geometry of Banach Spaces, vols.~1 and~2, Publ. Math. Univ. Paris VII, Paris (1983), 81--91.
\bibitem{sobczyk} A. Sobczyk, \textit{Projection of the space $m$ on its subspace $c_0$}, Bull. Amer. Math. Soc. ~47 (1941), 938--947.
\bibitem{ulam_1} S.M. Ulam, \textit{A Collection of Mathematical Problems}, Wiley-Interscience 1960.
%\bibitem{ulam_2} S.M. Ulam, \textit{Problems in Modern Mathematics}, Wiley 1964.
\bibitem{yost (JL space)} D. Yost, \textit{The Johnson--Lindenstrauss space}, Extracta Math.~12 (1997), 185--192.
%\bibitem{yost (different JL space)} D. Yost, \textit{A different Johnson--Lindenstrauss space}, New Zealand J. Math.~37 (2008), 47--49.
\end{thebibliography}
\end{document}